\documentclass[12pt]{amsart}
\usepackage{amscd,amssymb,graphics}

\usepackage{amsfonts}
\usepackage{amsmath}
\usepackage{amsxtra}
\usepackage{latexsym}
\usepackage[mathcal]{eucal}

\usepackage{graphics,colortbl}

\input xy
\xyoption{all}
\usepackage{epsfig}

\usepackage[pdftex,bookmarks,colorlinks,breaklinks]{hyperref}

\newtheorem{theorem}{Theorem}[section]
\newtheorem{fact}[theorem]{Fact}
\newtheorem{corollary}[theorem]{Corollary}
\newtheorem{lemma}[theorem]{Lemma}
\newtheorem{proposition}[theorem]{Proposition}

\newtheorem{question}[theorem]{Question}
\newtheorem{definition}[theorem]{Definition}
\numberwithin{equation}{section}
 
\newtheorem{prob}[theorem]{Problem}  

\theoremstyle{remark}
\newtheorem{remark}[theorem]{Remark}
\newtheorem{example}[theorem]{Example}

\newcommand{\ben}{\begin{enumerate}}
\newcommand{\een}{\end{enumerate}}
\newcommand{\bit}{\begin{itemize}}
\newcommand{\eit}{\end{itemize}}

\def\R {{\Bbb R}}
\def\Q {{\Bbb Q}}
\def \F {{\Bbb F}}

\def\N{{\Bbb N}}
\def\T{{\Bbb T}}

\def\Z {{\Bbb Z}}

\def\F{{\Bbb F}}

\def\Aut{{\mathrm Aut}\,}

\def\eps{{\varepsilon}}

\def\QED{\nobreak\quad\ifmmode\roman{Q.E.D.}\else{\rm Q.E.D.}\fi}

\def\a {\alpha}
\def\s {\sigma}
\def\g {\gamma}
\def\t {\tau}
\def\sk {\vskip 0.3cm}

\def\PGL{\operatorname{PGL}}

\def\GL{\operatorname{GL}}

\def\UT{\operatorname{UI}}
\def\UT{\operatorname{UT}}

\begin{document}

\title[]{Key subgroups in topological groups} 

		\author[]{M. Megrelishvili}
		\address[M. Megrelishvili]
	{\hfill\break Department of Mathematics
		\hfill\break
		Bar-Ilan University, 52900 Ramat-Gan	
		\hfill\break
		Israel}
		\email{megereli@math.biu.ac.il}

\author[]{M. Shlossberg}
\address[M. Shlossberg]
{\hfill\break School of Computer Science
	\hfill\break Reichman University, 4610101 Herzliya 
	\hfill\break Israel}
\email{menachem.shlossberg@post.runi.ac.il}



 \date{2024,  
 	October 2}   
\thanks{This research was supported by a grant of the Israel Science Foundation (ISF 1194/19) and also by the Gelbart Research Institute at the Department of Mathematics, Bar-Ilan  University}

\begin{abstract}  	   
	We introduce two minimality properties of subgroups in topological groups. A subgroup $H$ is a \textit{key subgroup} (\textit{co-key subgroup}) of a topological group $G$ if there is no strictly coarser Hausdorff group topology on $G$ which induces on $H$ (resp., on the coset space $G/H$) the original topology. Every co-minimal subgroup is a key subgroup while the converse is not true. Every locally compact co-compact subgroup is a key subgroup (but not always co-minimal). Any relatively minimal subgroup is a co-key subgroup (but not vice versa). Extending some results from \cite{MEG95,DM10} concerning the generalized Heisenberg groups, we prove that the center ("corner" subgroup) of the upper unitriangular group $\mathrm{UT(n,K)}$, defined over a commutative topological unital ring $K$, is a key subgroup. Every "non-corner" 1-parameter subgroup $H$ of $\mathrm{UT(n,K)}$ is a co-key subgroup. We study injectivity property of the restriction map $$r_H \colon \mathcal{T}_{\downarrow}(G) \to \mathcal{T}_{\downarrow}(H), \ \sigma \mapsto \sigma|_H$$ and show that it is an isomorphism of sup-semilattices for every central co-minimal subgroup $H$, where $\mathcal{T}_{\downarrow}(G)$ is the semilattice of coarser Hausdorff group topologies on $G$.
\end{abstract}	

\maketitle
	
\setcounter{tocdepth}{1}
	\tableofcontents 

\section{Introduction} 

A Hausdorff topological group $G$ is \emph{minimal} \cite{S71,Doitch} if every continuous isomorphism $f \colon G\to H$, with $H$ a topological group, is a topological  isomorphism (equivalently, if $G$ does not admit a strictly coarser Hausdorff group topology). 

There are many papers demonstrating that minimal groups are important in topological algebra, analysis and geometry. See for example, the survey paper \cite{DM14}. Sometimes, a wider look is required 
in relaxing the definition of minimality.

Recall 
\cite{MEG04} (resp., \cite{DM10}) 
that a subgroup $H$ of $G$ is said to be \emph{relatively minimal} (resp., \emph{co-minimal}) in $G$ if every coarser Hausdorff group topology on $G$ induces on $H$ (resp., on the coset space $G/H:=\{gH: g \in G\}$) 
the original topology. 

\sk 

The following two definitions are natural generalizations 
(see Proposition \ref{p:co-minIsKey}) 
of co-minimality and relative minimality, respectively.  

\begin{definition} \label{d:key}
	Let $H$ be a subgroup of a topological group $(G,\gamma)$. 
	\begin{enumerate} 
		\item 	We say that $H$ is a 
		\textbf{key subgroup} of $G$ if for every coarser Hausdorff group topology $\gamma_1 \subseteq \gamma$ on $G$ the coincidence of the two topologies on $H$ (that is,  $\gamma_1|_H = \gamma|_H$) 
		 implies that  $\gamma_1 = \gamma$. 
	
		\item 
		We say that $H$ is a 
		\textbf{co-key subgroup} of $G$ if for every coarser Hausdorff group topology $\g_1 \subseteq \g$ on $G$ satisfying $\g_1/H=\g/H$ it holds  that  $\g_1 = \g$.
	\end{enumerate} 
\end{definition}

In the sequel, $\g/H$ denotes the quotient topology on $G/H$. 
The concept of key minimal subgroups appears implicitly in some earlier publications \cite{MEG95, DM10, MS-Fermat}.  
See also Theorem \ref{examples2}.1 below. 

	Observe that a topological group is minimal if and only if it contains a relatively minimal key subgroup 
if and only if it contains a co-minimal co-key subgroup.

By Proposition \ref{p:cocompact}, 
locally compact co-compact subgroups are key subgroups (e.g.,  $\Z$ in $\R$). 
We give several distinguishing examples 
which confirm that the class of key subgroups is 
much larger than the class of all co-minimal subgroups.  
See Examples 
\ref{ex:R/Z},  \ref{ex:mnosm}, \ref{ex:zinq} and \ref{p:oneMore}  
below.

It is more difficult  to find co-key subgroups $H$ which are 
not relatively minimal. By Proposition \ref{p:co-key-noNorm}, 
the subgroup $H$ cannot be \textit{normal}.
A closed subgroup $H$ of a locally compact \textbf{abelian} group $G$ is co-key
	if and only if $H$ is relatively minimal if and only if $H$ is compact. 
For some natural sources of distinguishing examples in this case see 
\ref{prop: knoco} and \ref{ex:mnosm}.


Recall that for every separated continuous 
biadditive  mapping $w \colon E \times F \to A$ 
(where $E, F$ and $A$ be abelian topological groups) we have 
the generalized Heisenberg group $
H(w)= (A \oplus E) \rtimes F
$ 
which intuitively can be represented as a generalized matrix group 
	$$ H(w):=
\left( \begin{array}{ccc}
	1 & F & A \\
	0 & 1 & E \\
	0 & 0 & 1 
\end{array}\right).  
$$
According to Theorem \ref{t:AisKey},  the center 
$A$ is a key subgroup of $H(w)$ 
if and only if $w$ is 
minimal (see Definition \ref{d:smin}.b) 
 if and only if  $E$ and $F$ are co-key subgroups in $H(w)$. 
Moreover, by Theorem \ref{t:AisCo-Min}, $A$ is a co-minimal subgroup of $H(w)$ if and only if $w$ is strongly minimal (Definition \ref{d:smin}.c) if and only if $E$ and $F$ are relatively minimal in $H(w)$.

One of the consequences of these results
 is an example of a discrete matrix group 
 with a key subgroup 
which is not co-minimal and with a co-key subgroup which is not relatively minimal. 
  (see Corollary \ref{ex:mnosm} for details). 

In \cite{MS-Fermat} we study 
 minimality properties of matrix groups over subfields of local fields. 
 In the present paper we obtain some results regarding key subgroups and co-minimality in matrix groups. In particular, we show that the center of the upper unitriangular group  $\UT(n,K)$ (that is, the corner subgroup
 $\UT^{n-2}(n,K)$), defined over a commutative topological  unital ring $K$, is a key subgroup (Theorem \ref{t:1}). If the multiplication map $m \colon K\times K\to K$ is strongly minimal (e.g., when $K$ is an archimedean absolute valued field or a  local field), then this center is even co-minimal in $\UT(n,K)$ (Theorem \ref{t:2}). 
	
	In Theorem \ref{t:co-1} 
	we prove that every non-corner "1-parameter" subgroup 
	is a co-key subgroup of $G=\UT(n,K)$. 
	
	\sk  
	\noindent \textbf{Notation}. 
	For a topological group $(G,\g)$ denote by $\mathcal{T}_{\downarrow}(G)$ 
	the sup-semilattice of all coarser Hausdorff group topologies $\t \subseteq \g$ on $G$.
	
	\sk 
	
	For every subgroup $H$ of $G$ we have the following  natural order preserving restriction map (which is onto for every central subgroup $H$): 
	$$r_H \colon \mathcal{T}_{\downarrow}(G) \to \mathcal{T}_{\downarrow}(H), \ \s \mapsto \s|_H.$$  
	In Section \ref{s:inj} we study injectivity of $r_H$. 
	We say that $H$ is an \textit{inj-key} subgroup in $G$ if $r_H$ is injective. 
	We show in Theorem \ref{p:inj-key=co-min for central}.2 that co-minimal subgroups are inj-key. Hence, 
	$r_H$ is an isomorphism of semilattices for central co-minimal subgroups.
By Theorem \ref{t:inj-key is co-comp in LCA} a subgroup of a discrete abelian group is co-minimal  if and only if it is
 unconditionally closed of finite index.  
	
	 The restriction on the center 
	 is an isomorphism in the following remarkable cases (see  
	 Example \ref{ex:inj-keyCorHeis} for additional instances) : 
	$$r_{\F} \colon \mathcal{T}_{\downarrow}(\UT (n,\F)) \to \mathcal{T}_{\downarrow}(\F)$$ 
	$$r_{\F^{\times}} \colon \mathcal{T}_{\downarrow}(\GL(n,\F)) \to \mathcal{T}_{\downarrow}(\F^{\times})$$

In 
Section \ref{s:counterexample} we describe an example of a subgroup $H$ of a topological group $G$ such that $H$ is inj-key but not co-minimal. For simplicity we do not include the proof. We are going to give 
the details in a separate article.  

	Below we pose 
	open Questions \ref{q:all nontr s are key}, \ref{q:no key s}, \ref{q:both}, \ref{q:center-key}, 
	 \ref{q:key}.  
	
	\vskip 0.3 cm
\noindent \textbf{Acknowledgment.} 
We are very grateful to the referee who suggested several important improvements. Among others, Proposition \ref{p:cocompact}.1 was inspired by 
his question. 
We also thank Vladimir Pestov for 
 proposing to examine topological groups with metrizable universal minimal flow (see Section \ref{s:counterexample}).

%

\sk 
\section{Some properties of key and co-key subgroups}   

Unless otherwise is stated, the topological groups in this paper are assumed to be Hausdorff.  
For a topological group $G$ denote by
 $\mathcal S(G), 
 \mathcal K(G) $ and  $\mathcal {CK}(G)$ the poset of all subgroups of $G$ and its subposets of all 
 key and co-key subgroups, respectively. 
It always holds that $G  \in \mathcal K(G)$ and $\{e\}\in \mathcal {CK}(G).$

We list here some properties of key subgroups,  where 
assertions 
(1)-(3) are straightforward. 


\begin{lemma} \label{examples1} Let $H_2 \leq H_1$ be subgroups of a topological group $G.$  
	\ben  
	
	\item 	
	$H_2 \in 
	 \mathcal {K}(G) \Rightarrow H_1 \in 
	 \mathcal {K}(G)$ (so, $\mathcal K(G)$ is up-closed in the poset $\mathcal S(G)$). 
			\item 	If $H_2\in \mathcal K(H_1)$  and $H_1\in \mathcal K(G)$, then $H_2\in \mathcal K(G).$
		This can be easily extended to every finite chain 
		$H_m \leq H_{m-1} \leq \cdots \leq H_1$ of subgroups in $G$. 
		
			\item 	
		$\mathcal K(G)$ contains a relatively minimal subgroup if and only if $G$ is minimal. This occurs precisely when $\mathcal K(G)= \mathcal S(G).$  
		\een
\end{lemma}

Since the class of minimal groups is closed under taking 
closed central subgroups (see, for example, \cite[Proposition 7.2.5]{DPS89}) we immediately obtain: 

\begin{lemma}\label{lem:ceniskey}
	Let $G$ be a topological group. Then $G$ is minimal if and only if $Z(G)$ is a minimal key subgroup. 
\end{lemma}



The following useful result is known as \textit{Merson's Lemma} (\cite[Lemma 7.2.3]{DPS89} or \cite[Lemma 4.4]{DM14}). 	

\begin{fact} \label{firMer} (Merson's Lemma)
	Let $(G,\gamma)$ be a (not necessarily Hausdorff) topological group and $H$ be a subgroup of $G.$ If 
$\gamma_1\subseteq \gamma$ is a coarser group topology on $G$ 
	such that $\gamma_1|_{H}=\gamma|_{H}$ and $\gamma_1/H=\gamma/H$, then $\gamma_1=\gamma.$
\end{fact}

\begin{corollary} \label{c:key s reform} 
	Let $H$ be a subgroup of a topological group $(G,\gamma)$. 
	\begin{enumerate}
		\item 
			The following conditions are equivalent:
		\begin{enumerate}
			\item $H$ is a key subgroup of $G$.  
			\item For every  
			$\gamma_1 \in \mathcal{T}_{\downarrow}(G)$ the condition $\gamma_1|_H=\gamma|_H$ implies the coincidence of quotient topologies $\gamma_1 /H =\gamma /H$.   
	\end{enumerate}
	
\item 
The following conditions are equivalent:   
	\begin{enumerate}   
	\item $H$ is a co-key subgroup of $G$.  
	\item For every 
	$\gamma_1 \in \mathcal{T}_{\downarrow}(G)$ the condition $\gamma_1/H=\gamma/H$ implies the coincidence of subspace topologies $\gamma_1|_H=\gamma|_H$.   
\end{enumerate}
	\end{enumerate}
\end{corollary}

Recall that by \cite[Lemma 3.5.3]{DM10} a dense subgroup is always co-minimal.
\begin{proposition} \label{p:co-minIsKey}  
	Let $G$ be a topological group.	Then 
	\begin{enumerate} 
		\item every co-minimal
		(e.g., dense) subgroup of $G$ is a key subgroup;  
		\item  every relatively minimal (e.g., compact) subgroup of $G$ is co-key. 
	\end{enumerate}

\end{proposition}
\begin{proof} 
	(1) 
	Let $H$ be a co-minimal subgroup of $(G,\g)$ and assume that $\gamma_1|_{H}=\gamma|_{H}$, where 
	$\gamma_1 \in \mathcal{T}_{\downarrow}(G)$. Since $H$ is co-minimal in $G$, we have $\gamma_1/H=\gamma/H$. 
	Now, apply Corollary \ref{c:key s reform}.1. 
	
	\sk 
	(2) 
	Similar 
	proof using Corollary \ref{c:key s reform}.2 this time.  
\end{proof}

%

\begin{corollary} \label{c:clos} 
	$H\in \mathcal K(G)$ if and only if its closure 
	$cl(H)\in \mathcal K(G).$ 
\end{corollary}
\begin{proof}
	If $H\in \mathcal K(G)$, then $cl(H)\in \mathcal K(G)$ by Lemma \ref{examples1}.1. 
	
	 Conversely, let us assume that  $cl(H)\in \mathcal K(G).$ Being dense in $cl(H),$ the subgroup $H$ must be co-minimal in its closure, in view of  \cite[Lemma 3.5.3]{DM10}. By Proposition \ref{p:co-minIsKey}.1, $H\in \mathcal K(cl(H))$. Using Lemma \ref{examples1}.2 we deduce that $H\in \mathcal K(G).$
\end{proof}



\begin{fact} \label{r:RD} 
	\cite[page 102]{RD} 
Let $G\leq H$ be subgroups of a topological group $(X.\t)$.  
	\begin{enumerate}
		\item  
		Then we have the corresponding coset spaces $X/G$, $X/H$ and 
	the	canonical open quotient maps such that the following diagram is commutative 
		\begin{equation*}
			\xymatrix { X \ar[dr]_{q_H} \ar[r]^{q_G} & X/G
				\ar[d]^{p} \\
				&  X/H \ \ & }
		\end{equation*}
		where 
		$p \colon X/G \to X/H, \ p(xG)=xH$ is a 
		continuous open surjection;
		\item 
		\cite[Lemma 5.32]{RD} \ $(\t |_H)/G=(\t/G)|_{q_G(H)}$. 
	\end{enumerate}
\end{fact}

%



\sk 
\begin{lemma} \label{l:LargerSubAlsoCoKey}  
	Let $G\leq H$ be subgroups of a topological group $(X.\t)$.  
	\begin{enumerate}
		\item  $H\in \mathcal{CK}(X) \Rightarrow G\in \mathcal{CK}(X)$  
		 (so, $\mathcal{CK(X)}$ is down-closed in the poset $\mathcal S(X)$). 
		\item  $G\in \mathcal{CK}(H) \Rightarrow G\in \mathcal{CK}(X).$
		\item 	$\mathcal {CK}(G)$ contains a co-minimal 
	(e.g., a dense) subgroup of $G$  if and only if $G$ is minimal. This occurs precisely when $\mathcal {CK}(G)= \mathcal S(G).$
		
	\end{enumerate} 
\end{lemma}
\begin{proof} (1) 
	Let $\s \in \mathcal{T}_{\downarrow}(X,\t)$ such that $\s / G = \t /G$. We have to show that $\s=\t$. 
	First of all observe that $\s /H= \t /H$. 
	Indeed, since $G$ is a subgroup of $H$, in virtue of Fact \ref{r:RD}.1, the quotient topology $\s/H$ is the same as the quotient topology on $X/H$ induced from the open map $p \colon (X/G, \s/G) \to (X/H,\s/H)$. Similarly the quotient topology $\t/H$ is the same as the quotient topology on $X/H$ induced from the open map $p \colon (X/G, \t/G) \to (X/H,\t/H)$. So, the equality $\s / G = \t /G$ ensures that $\s /H= \t /H$.
		If $H$ is a key-subgroup of $X$, then by definition $\s=\t$.

	(2) 
	Now, suppose that  $G\in \mathcal{CK}(H).$ Let 
	$\s \in \mathcal{T}_{\downarrow}(X,\t)$ such that $\s / G = \t /G$. We have to show that $\s=\t$. 
	As in the proof of (1) observe that $\s /H= \t /H$.

	  By Merson's Lemma it is enough to show that $\s|_H=\t|_H$.  
	Since $G$ is co-key in $H$, we have only to prove that $(\s|_H)/G=(\t|_H)/G$. 
	As $\s/G=\t/G$, then, in particular,   $(\t/G)|_{q(H)}=(\s/G)|_{q(H)},$ where $q:=q_G\colon X\to X/G$ is the canonical quotient.  By 
	Fact \ref{r:RD}.2, 
	it holds that $(\t|_H)/G=(\t/G)|_{q(H)}$ and $(\s|_H)/G=(\s/G)|_{q(H)}$. Therefore, $(\t|_H)/G=(\s|_H)/G$ and our proof is completed. 	
\end{proof} 

Let $\a \colon G \times X \to X$ be a continuous action of $G$ on $X$ by group automorphisms and $M:=X \rtimes_{\a} G$  be a topological semidirect  product. 
We have canonical embedding of topological groups $i_X \colon X \hookrightarrow M, x \mapsto (x,e_G)$ and $i_G \colon G \hookrightarrow M, g \mapsto (e_X,g)$, where $e_X$ and $e_G$ are the neutral elements in $X$ and $G$ respectively. 
We will often identify $X$ with the normal subgroup $X \times\{e_G\}$ of $M$ and $G$ with the subgroup $\{e_X\} \times G$. 

\begin{fact} \label{f:FinalTop} \cite[Prop. 6.15.b]{RD} 
	The topological semidirect product 
	$M:=X \rtimes_{\a} G$ carries the final (strongest) group topology with respect to the inclusion maps $i_X$ and $i_G$.
\end{fact}

We have a continuous (open) retraction of topological groups $M \to G, (x,g) \mapsto g$ and a continuous (open) map $q_G \colon M \to M/G$, $(x,g) \mapsto (x,g)G$.  
Recall that its restriction  
$$
q_G|_X \colon X   \to M/G, \ \ x \mapsto xG
$$
is a homeomorphism (see \cite[Proposition 6.17.a]{RD}). 
Denote by $h$ the converse homeomorphism $h \colon M/G \to X$. 
Then $h(q_G(x))=x$ for every $x \in X$. 

Following \cite{MEG95} we say that the action $\a \colon G \times X \to X$ is \textit{t-exact} (that is, \textit{topologically exact}) if there is no strictly coarser 
(not necessarily Hausdorff) group topology on $G$ such that the action remains continuous. 
\begin{theorem} \label{examples2}  
 Let $\a \colon G \times X \to X$ be a continuous action of $G$ on $X$ by group automorphisms  and $M:=X \rtimes_{\a} G$  be a topological semidirect  product. If $X$ is abelian and the action $\a$ is t-exact, then 
 \begin{enumerate}
 	\item $X$ is a key subgroup of $X \rtimes_{\a} G$.  
 	\item $G$ is a co-key subgroup of $X \rtimes_{\a} G$.  
 \end{enumerate} 
\end{theorem}
\begin{proof} 
	(1) 
	 Let $\gamma_1 \in \mathcal{T}_{\downarrow}(M,\g)$, where 
	  $M=X \rtimes_{\a} G$.  
	 By \cite[Prop. 2.6]{MEG95} 
	the following induced action is continuous
	$$
	(G,\gamma_1/X) \times (X, \gamma_1|X) \to (X, \gamma_1|X). 
	$$ 
	  Assume that $\g_1|_X=\g|_X$. Then by the t-exactness of the original action $\a$, the topology $\g_1/X$ on $G$ coincides with the given topology (which is $\g/X$). Therefore, by Merson's Lemma we conclude 
	  that $\g_1=\g$.  

	\sk 
	(2)  
	Now suppose that 
 $\gamma_1/G =\gamma/G$ on $X$. In order to prove that $G$ is a co-key subgroup of $M$ we have to show that $\gamma_1=\gamma$. 
	
	Since $X$ is a key subgroup of $M$ by 
	(1), it is enough to prove that $\gamma_1|X =\gamma|X$. 
	The quotient map $(M,\gamma_1) \to (M/G,\gamma_1/G)$ is trivially continuous. 
	Hence also the map $q_G \colon (M,\gamma_1) \to (M/G,\gamma/G)$  (
	as $\gamma_1/G =\gamma/G$). 
	
	Consider the homeomorphism $h \colon M/G \to X$ described 
	 just before this proposition. Then the restriction of the composition $h \circ q_G$ induces on the subspace $(X,\g_1|_X)$ the continuous identity map $id \colon (X,\g_1|_X) \to (X,\g|_X)$. Hence, $\gamma_1|X=\gamma|X$, as desired.      	
\end{proof}

Note that assertion (1) is an easy reformulation of \cite[Corollary 2.8]{MEG95} in terms of key subgroups. For a more general example with a t-exact system of actions and central retractions, see \cite[Proposition 2.7]{MEG95}. 


\sk 
A (closed) subgroup $H$ of a topological group $G$ is called \textit{co-compact} if the coset space $G/H$ is compact and Hausdorff. Co-compact subgroups play a major role especially in the theory of Lie groups and geometry (co-compact lattices among others).  

Recall 
 that a topological group $G$ is \textit{Raikov complete} \cite{Raikov} if and only if it is complete in its two-sided uniformity if and only if it is closed in any topological group containing $G$ as a subgroup. 
 For example, every locally compact group is Raikov complete. 
 
 \begin{proposition} \label{p:cocompact}  
 	Let $H$ be a 
 	subgroup of a topological group 
 	$(G,\g)$ such that 
 	$cl(H)$ is Raikov complete. 
 	%
 	
 	\begin{enumerate}
 		\item 
 		If $G/H$ is compact,  then $H \in \mathcal {K}(G)$. 
 		\item If $H$ is a normal subgroup of $G$ and the factor group $G/cl(H)$ is minimal, then $H \in \mathcal {K}(G)$.   
 	\end{enumerate} 
 	 \end{proposition}	 
 \begin{proof} (1)  Let $M:=cl(H)$. In view of Corollary \ref{c:clos}, it suffices to prove that   $M \in \mathcal {K}(G)$. As $G/H$ is compact and $M$ is a closed subgroup of $G$ containing $H$ it follows from Fact \ref{r:RD}.1 that $M$ is co-compact.
 Let $\g_1 \in \mathcal{T}_{\downarrow}(G,\g)$ such that 
  $\gamma_1|_M=\gamma|_M$. Since $M$ is a co-compact subgroup of $G$ the coset $G$-space $G/M$ is compact with respect to the quotient topology $\gamma/M$. Clearly, $\gamma_1 /M \subseteq \gamma /M$ is also compact. Moreover,  $\gamma_1 /M$ is Hausdorff because $M$, being Raikov complete in its topology $\gamma_1|_M=\gamma|_M$,
  is also $\gamma_1$-closed in $G$. Therefore, $(G/M, \gamma_1/M)$ is also Hausdorff. It follows that $\gamma_1/M=\gamma/M$. Now, by Merson's Lemma (Fact \ref{firMer}) one conclude that $\gamma_1 =\gamma$.   
  
  (2) The proof is similar to the case of (1). 
 \end{proof}


For  abelian groups $G$ Proposition \ref{p:cocompact}  can be found (up to reformulations) in \cite[Theorem 3.21]{HPTX19}. 

\begin{example} \label{ex:R/Z} \ 
	The subgroup $\Z^n$ is a key subgroup of $\R^n$ by Proposition \ref{p:cocompact}.1 (for every $n \in \N$). 
		However, we will show that $\Z^n$ is not co-minimal in $\R^n$ as it follows from Example \ref{r:super-rem}.1 below. 
		%
\end{example}
By results of B. Kadri \cite{Kadri} 
a non-compact locally compact group 
every non-trivial subgroup of which is co-compact 
is isomorphic either to $\Z$ or to $\R$. Since every locally compact co-compact subgroup is a key subgroup (Proposition \ref{p:cocompact}.1) the following 
question is very natural. 
\begin{question} \label{q:all nontr s are key} 
	Study the non-compact locally compact groups $G$
	every non-trivial subgroup of which is a key (co-minimal) subgroup. What if $G$ 
	 is abelian?
\end{question}
 \begin{proposition} \label{p:allKey-in-R}
 Any non-trivial subgroup of $\R$ is a key subgroup. 
 \end{proposition}
\begin{proof}
Let $G$ be a non-trivial subgroup of $\R$. If $G=x\Z$ for some $x\in \R$, then $G$ is a key  subgroup by Proposition \ref{p:cocompact}.1. Otherwise, it is known that $G$ must be dense in $\R$. So, it is a key subgroup by 
Corollary \ref{c:clos}.  
\end{proof}
 Another extreme case is a topological group $G$ without any proper key subgroups 
 (i.e, $\mathcal K(G)=\{G\}$). 
 The following question was inspired by the referee. 
 \begin{question} \label{q:no key s} 
 		Study the locally compact groups $G$ such that any key (co-minimal) subgroup is necessarily $G$. What if $G$ 
 		 is abelian ?
 \end{question} 
Recall that a \textit{topologizable}  group  is an abstract group admitting a non-discrete group topology. 
It is equivalent to say that this group  is not minimal when equipped with the discrete topology. 
By \cite[Theorem 1.4]{KOO}, for every sufficiently large odd $n\in \N$ there exists a topologizable 
Tarski 
Monster (i.e., an infinite simple group with all proper subgroups finite cyclic) of exponent $n$. By \cite{DM10},  the    quasicyclic Pr\" ufer  $p$-groups $\mathbb{Z}(p^\infty)$ admit no minimal group topologies. 

Let $G\in\{\mathbb{Z}(p^\infty), M\}$, where  $M$ is a topologizable Tarski Monster   endowed with the discrete topology and $\mathbb{Z}(p^\infty)$ is equipped with any group topology.    
 Then $G$ is not minimal and has all proper subgroups finite. It follows that $\mathcal K(G)=\{G\}$ 
 while $\mathcal{CK}(G)=\mathcal S(G)\setminus\{G\}.$
 

%

 \begin{proposition} \label{p:co-key-noNorm} 
 	Let $(G,\t)$ be a topological group and $H$ be its \textbf{normal} (not necessarily, closed) subgroup. 
 	The following conditions are equivalent:
 	\begin{enumerate}
 		\item $H$ is co-key in $G$;  
 		\item $H$ is relatively minimal in $G$.
 	\end{enumerate}
 \end{proposition} 
 \begin{proof}
 	(1) $\Rightarrow$ (2): Since $H$ is a normal subgroup, the quotient topological group $(G/H,\t/H)$ is well defined. 
 	Denote by $q^{-1}(\t/H)$ the preimage group topology on $G$ induced by the quotient $q \colon G \to (G/H,\t/H)$.
 	Assuming the contrary, let $H$ be not relatively minimal in $G$.    
 	Then there exists $\s \in \mathcal{T}_{\downarrow}(G,\t)$  
 	 such that
 	$\s|_H$ is strictly coarser than the original topology $\t|_H$ on $H$.
 	Consider the supremum $\nu:=\sup\{\s,q^{-1}(\t/H)\}$ of group topologies. Then
 	 $\s \subseteq \nu \subseteq \t$ 
 	  and $\nu \in \mathcal{T}_{\downarrow}(G,\t)$. 
 	Also, $q^{-1}(\t/H) / H = \t/H$. Since  $q^{-1}(\t/H) \subseteq \nu \subseteq \t$ we have 
 	$\nu/H=\t/H$ on $G/H$. It is easy to see that  $\nu|_H=\s|_H \neq \t|_H$. 
 	This implies that $\nu \neq \t$. Therefore, $H$ is not a co-key subgroup in $G$.  
 	
 	(2) $\Rightarrow$ (1) Is just Proposition \ref{p:co-minIsKey}.2.
 \end{proof}
 
 \begin{corollary} \label{c:both} \
 	\begin{enumerate}
 		\item 	A topological group is minimal if and only if it contains a normal subgroup which is key and co-key at the same time. 
 		\item 	A non-minimal  abelian topological group does not contain a subgroup which is key and co-key at the same time. 
 		 \item 
 		A closed subgroup $H$ of a (locally compact) \textbf{abelian} group $G$ is co-key
 		if and only if $H$ is relatively minimal if and only if $H$ is (compact) minimal. 
 	\end{enumerate} 
 \end{corollary}
 
For (3) we use 
Fact \ref{f:strCl-DM}.2. 
 In particular, every non-trivial closed subgroup of $\R$ fails to be co-key (being, in contrast, 
 a key subgroup by Proposition \ref{p:allKey-in-R}).  The same is true for every non-trivial 
 subgroup of $\Z$. 

 It is unclear if the normality assumption can be removed from Corollary \ref{c:both}.1.  We leave open the following question.
 \begin{question} \label{q:both} 
 	Let $G$ be a topological group such that $\mathcal K(G) \cap \mathcal {CK}(G)\neq \emptyset.$ Must $G$ be minimal?
 \end{question} 
 
\section{Generalized Heisenberg groups} \label{s:GH}

%
We recall a natural generalization of the Heisenberg group (see, for example, \cite{MEG95,DM10,DM14}) which is based on biadditive mappings.  Let $E,F,A$ be abelian groups. A map $w: E \times F \to A$ is said to be {\it biadditive} if the induced mappings
$$
w_x: F \to A, w_f: E \to A, \ \  w_x(f):=w(x,f)=:w_f(x)
$$
are homomorphisms for all $x \in E$ and $f \in F$. We say that $w$ is {\it separated} if the induced homomorphisms separate points.
That is, for every non-zero $x_0 \in E, f_0 \in F$ there exist $f \in F, x \in E$ such that $f(x_0) \neq 0_A, f_0(x) \neq 0_A$, where $0_A$ is the zero element of $A$.
\begin{definition} \label{d:H(w)}
	{\em  Let $E, F$ and $A$ be abelian topological groups and $w \colon E \times F \to A$
		be a continuous biadditive mapping.
		Define the induced action of $F$ on $A \times E$ by
		$$w^{\nabla}: F \times (A \times E) \to (A \times E), \ \ \ w^{\nabla}(f,(a,x)) =(a+f(x),x).$$
In this way, $F$ acts by group automorphisms on $A \times E$.  
Denote by
$$
H(w)= (A \oplus E) \rtimes F
$$
the corresponding semidirect product of $F$ and the group $A \times E$. The resulting group, as a topological space, is the product $A \times E \times F $.
This product topology will be denoted by $\gamma$. The group operation is defined by the following rule: for a pair
$$
u_1=(a_1,x_1,f_1), \hskip 0.4cm u_2=(a_2,x_2,f_2)
$$
define
$$u_1 \cdot u_2 = (a_1+a_2+f_1(x_2), x_1+x_2,
f_1 +f_2)$$ where, $f_1(x_2)=w(x_2,f_1)$.}
\end{definition}
If  $w$ is separated,  then $H(w)$ becomes a two-step nilpotent topological group and $Z(H(w))=A$. We call $H(w)$ the {\it generalized Heisenberg group} induced by $w$.
Intuitively we can describe this group in the matrix form
	$$ H(w):=
\left( \begin{array}{ccc}
	1 & F & A \\
	0 & 1 & E \\
	0 & 0 & 1 
\end{array}\right). 
$$

Elementary computations for the commutator $[u_1,u_2]$ give
$$
[u_1,u_2] = u_1u_2u_1^{-1}u_2^{-1}= (f_1(x_2)-f_2(x_1),0_E,0_F).
$$

Often we will identify $E$ with $\{0_A \} \times E \times \{0_F \}$, $F$ with $\{0_A\} \times \{0_E\} \times F$ and $E \times F$ with $\{0_A\} \times E \times F$ .

In the case of a normed space $X$ and the canonical bilinear function
$w: X \times X^* \to \R$ we write $H(X)$ instead of
$H(w)$. Clearly, the case of $H(\R^n)$ (induced by the scalar product $w: \R^n \times \R^n \to \R$) gives the classical
2n+1-dimensional Heisenberg group.

\begin{definition} \label{d:smin}  
Let $(E,\s)$, $(F, \tau)$, $(A,\nu)$ be abelian  groups such that the separated biadditive mapping
		$$
		w \colon (E,\s) \times (F,\tau) \to (A,\nu) \eqno(*)
		$$
		is continuous.
		\begin{itemize} 
			\item[(a)] A triple $(\s_1,\t_1,\nu_1)$ of coarser Hausdorff group topologies 
			$\s_1 \subseteq\s$, $\t_1 \subseteq \t$, $\nu_1 \subseteq \nu$ on $E, F$ and
			$A$, respectively, is called  {\em compatible}, if
			$$
			w: (E,\s_1) \times (F,\t_1) \to (A,\nu_1)
			$$
			is continuous.
			\item[(b)] \cite{MEG95} 
			 We say that the biadditive mapping (*) is {\it minimal} if $\s_1=\s, \t_1=\t$
			holds for every compatible triple $(\s_1,\t_1, \nu)$ (with $\nu_1:=\nu$).

			\item[(c)]   \cite{DM10} 
			 The biadditive mapping (*) is called {\em strongly minimal}
			if for every compatible triple $(\s_1,\t_1,\nu_1)$ it follows that $\s_1=\s, \t_1=\t$.
	\end{itemize} 
\end{definition}

\begin{remark} 
	The definition of a compatible triple $(\s_1,\tau_1,\nu_1)$  does not  require that $\s_1 $ and $\t_1 $ are Hausdorff. 
	Indeed,   since $w$ is separated and $\nu_1$
	is Hausdorff it follows that  $\sigma_1$ and
	$\tau_1$ are automatically Hausdorff.
\end{remark}


\begin{lemma} \label{l:min-map} \ 
	\ben
	\item Every strongly minimal map is minimal.
	\item Every compatible triple $(\s_1,\t_1,\nu_1)$ of topologies gives
	rise to the corresponding (product) topology $\g_1$ on the
	Heisenberg group $H(w)=(A \times E) \rtimes F$ which is a coarser
	Hausdorff group topology (that is, $\g_1  \in \mathcal{T}_{\downarrow}(G,\g)$).
	\item If the mapping $w$ is minimal and the group $A$ is minimal, then $w$ is strongly minimal.
	\een
\end{lemma}

 We  now provide some examples of (strongly) minimal biadditive maps.

\begin{example} \label{p:s-m-bi} \  \cite{MEG95, SH, DM14} 
	\ben
	\item The canonical biadditive mapping 
	 $G \times G^* \to \T,
	\hskip 0.2cm (g,\chi) \mapsto \chi(g)$ is strongly minimal for every locally compact abelian group G, where $G^*$ is the Pontryagin dual of $G$. 
	\item The canonical bilinear mapping $V\times V^* \to \R, \hskip 0.2cm (v,f) \mapsto f(v)$
	is strongly minimal for all normed spaces $V$ (where $V$ and its dual $V^*$ carry the norm topology).
	
	\item The multiplication map $m \colon A \times A \to A$ is minimal for every topological unital ring $A$. If $m$ is strongly minimal, then $A$ is necessarily  a minimal topological ring 
	(i.e., having no strictly coarser Hausdorff ring topology).
	\item More generally, for a topological unital ring $K$ define   
	$$w_n \colon K^n\times K^n\to K, \ w_n(\bar{x},\bar{y})= \sum_{i=1}^n x_iy_i.$$ 
 (a) Then $w_n$ is a  minimal biadditive mapping for every $n\in \N$. 
	
	\sk 
	Indeed, let $\nu$ be the given topology on $K$ and $\s=\t=\nu^n$ be the product topology on $K^n.$ To see that $w_n$ is minimal it suffices to show that if $\s_1\subseteq \s$ is a coarser group topology such that the triple $(\s_1,\t,\nu)$ is compatible then $\s_1=\s.$ To this aim, it is enough to prove that  the projection map $p_i\colon (K^n,\s_1)\to K$ is continuous for any $i\in \{1,\ldots, n\}.$  The continuity of this map follows from the $(\s_1,\t,\nu)$-continuity of $w_n$ taking into account that $w_n(\bar{x},e_i)=x_i$ for every $\bar{x}\in K^n$, where $e_i\in K^n$ is defined by $p_i(e_i)=1$ and $p_j(e_i)=0$ whenever $j\neq i.$ 
	
	(b) If $K$ is an archimedean absolute valued (not necessarily associative) division ring, then $w_n$ is even strongly minimal \cite{SH}.
	\een
	\end{example} 
	
	Regarding 
	 assertion (3) recall that every non-discrete locally retrobounded division ring is  minimal as a topological ring. Indeed, this follows from  \cite[Theorem 13.8]{Warner-R}. 
	 In particular, every local field is a minimal topological ring. 
	 Recall that a \textit{local field} is a non-discrete locally compact field.  
	
	\begin{remark} \label{switch} 
		Let $w \colon E \times F \to A$
		be a continuous separated biadditive mapping with the Heisenberg group $H(w)=(A \oplus E) \rtimes F$. 
		Consider the switched biadditive mapping 
		$$
		w_{\blacktriangleleft} \colon  F \times E \to A, \ w_{\blacktriangleleft}(f,x):=w(x,f)=f(x).
		$$ 
		Then we have the corresponding Heisenberg group  
		$H(w_{\blacktriangleleft})=(A \oplus F) \rtimes E$ and the following topological isomorphism 
		$$s \colon H(w) \to H(w_{\blacktriangleleft}), \ \ \  (a,x,f) \mapsto (f(x)-a,-f,-x).$$ 
	\end{remark}
	
\begin{theorem} \label{t:AisKey} 
	Assume that $w \colon E \times F \to A$ is a separated  biadditive map. 
	Then the following conditions are equivalent:
	\begin{enumerate}
		\item $A$ is a key subgroup of the generalized Heisenberg group $H(w)$. 
		\item $w$ is a minimal biadditive mapping.
		\item $E$ and $F$ are co-key subgroups in $H(w)$. 
	\end{enumerate}  	 
\end{theorem} 
\begin{proof} (1) $\Rightarrow$ (2):  	If $w$ is not minimal, then there exist Hausdorff group topologies $\s_1 \subseteq\s$, $\t_1 \subseteq \t$ on $E$ and $F$ respectively such that 
	$$
	w: (E,\s_1) \times (F,\t_1) \to (A,\nu)
	$$
	is continuous and $\s_1 \neq \s$ or $\t_1 \neq \t$. Then the corresponding Heisenberg group $H(w)= (A \times E) \rtimes F$ is well defined with respect to the product topology $(\nu \times \s_1 \times \t_1)$ which is strictly coarser than the original topology  $(\nu \times \s \times \t)$. However, both of them induce the same topology $\nu$ on the subgroup $A$. 
	
	\sk 
	(2) $\Rightarrow$ (1):   It is a reformulation of \cite[Proposition 2.9]{MEG95}, or of \cite[Lemma 5.9]{DM14}. 
	
	\sk 
	 (2) $\Rightarrow$ (3) 
	If $w$ is minimal then \cite[Lemma 2.2]{MEG95} guarantees that the both of the induced actions: (i) of $F$ on $A \oplus E$  and (ii) $E$ on $A \oplus F$ are t-exact. So, 
	Theorem 
	\ref{examples2}.2 
	is applicable. It implies that $F$ and $E$ are co-key subgroups in the corresponding semidirect products which both can be identified with $H(w)$.  	
	
	\sk 
	(3) $\Rightarrow$ (2)  	
	Let $\s_1 \subset \s$ be a strictly coarser Hausdorff group topology on $E$ such that 
	$w \colon  E \times F \to A$ is still continuous. Consider the corresponding semidirect product $(M,\g_1)=(A \oplus E) \rtimes F$. Then the quotient topology $\g_1/F$ on $A \oplus E$ is 
	the original topology 
	 (see \cite[Prop. 6.17 (a)]{RD}). 
	That is, $\g_1/F=\g / F$. On the other hand, $\g_1 \neq \g$. 
	This means that $F$ is not a co-key subgroup of $H(w)$. 
	
	Similar proof works for the subgroup $E$ of $H(w)$ using the 
	switched biadditive mapping 
	$w_{\blacktriangleleft} \colon  F \times E \to A$ and the isomorphism $s \colon H(w) \to H(w_{\blacktriangleleft})=(A \oplus F) \rtimes E$ from Remark \ref{switch} where 
	$s(E)=E$. 
\end{proof}

Clearly, if $A$ is a minimal group (e.g., compact) then Theorem \ref{t:AisKey}  ensures that  $H(w)$ is a minimal group for every minimal biadditive mapping $w$.  
One of the related 
applications (from \cite{MEG95}) is 
establishing the minimality of the Heisenberg group $H(w)$ modeled on the duality mapping 
$w \colon G \times G^* \to \T$ for every locally compact group $G$ (where $G^*$ is its dual group).

\begin{theorem} \label{t:AisCo-Min} 
	Assume that $w \colon E \times F \to A$ is a separated biadditive map.  
Then the following conditions are equivalent:
\begin{enumerate}
	\item $A$ is a co-minimal subgroup of $H(w)$. 
	\item $w$ is strongly minimal. 
	\item $E$ and $F$ are relatively minimal in $H(w)$. 
\end{enumerate} 
\end{theorem} 
\begin{proof} 
	(1) $\Rightarrow$ (2):	If $w$ is not strongly minimal then there exist Hausdorff group topologies $\s_1 \subseteq\s$, $\t_1 \subseteq \t$, $\nu_1 \subseteq \nu$ on 
	$E$, $F$ and $A,$  
	 respectively, such that 
	$$
	w: (E,\s_1) \times (F,\t_1) \to (A,\nu_1)
	$$
	is continuous and $\s_1 \neq \s$ or $\t_1 \neq \t$. Then the corresponding Heisenberg group $(H(w),\g_1)= (A \times E) \rtimes F$ is well defined with respect to the product topology $\g_1:=(\nu \times \s_1 \times \t_1)$ which is strictly coarser than the original topology  $\g:=(\nu \times \s \times \t)$. However, the new coset topology on $(H(w),\g_1)/A$ is strictly coarser than the original topology $(H(w),\g)/A$.

	\sk
		(2) $\Rightarrow$ (1) and 	(2) $\Rightarrow$ (3): 
		See \cite[Theorem 5.1]{DM10}. 
		
		\sk 
			(3) $\Rightarrow$ (2):  
			Very similar to the proof of (1) $\Rightarrow$ (2). 
			 Let $(\s_1,\t_1,\nu_1)$ be a compatible triple of coarser Hausdorff group topologies.  
 Assuming that $w$ is not strongly minimal. Then $\s_1 \neq \s$ or $\t_1 \neq \t$. 
The corresponding Heisenberg group $(H(w),\g_1)= (A \times E) \rtimes F$ is well defined with respect to  $\g_1:=(\nu \times \s_1 \times \t_1)$, 
 where at least one the subspace topologies on $E$ or on $F$ is strictly coarser than the original one.    
	\end{proof}
		
\begin{corollary}\label{prop: knoco}
	Assume that $w \colon E \times F \to A$ is a separated biadditive map. Then the following conditions are equivalent: 
	\begin{enumerate}
		\item $A$ is a key subgroup of $H(w)$ but not co-minimal. 
		\item $w$ is minimal but not strongly minimal. 
		\item $E$ and $F$ are co-key subgroups of $H(w)$ but at least one of them is not relatively minimal.  
	\end{enumerate}
\end{corollary}
\begin{proof}
	Combine Theorems \ref{t:AisKey} and \ref{t:AisCo-Min}. 
\end{proof}

	\begin{corollary} \label{ex:mnosm} 
			Consider the multiplication map $m \colon \Z \times \Z \to \Z$, where $\Z$ is the discrete ring of all integers and $H(m)=(\Z \times \Z) \rtimes \Z$ is the corresponding (discrete) Heisenberg group.  
			$$ H(m):=
			\left( \begin{array}{ccc}
				1 & \Z & \Z \\
				0 & 1 & \Z \\
				0 & 0 & 1 
			\end{array}\right).  
			$$	 
		Then 
		\begin{enumerate}
			\item  the center (corner-subgroup)
			 $A:=G_{1,3}$ (isomorphic to $\Z$) is a key subgroup in $H(m)$ but 
			  not co-minimal;  
			\item the 1-parameter non-corner integer subgroups 
			$G_{1,2}$, $G_{2,3}$ 
		(isomorphic to $\Z$) are co-key but not relatively minimal in $H(m)$. 
		\end{enumerate} 		
	\end{corollary}
\begin{proof} 	
	Clearly, $\Z$ is not a minimal ring (the $p$-adic topology is a ring topology). 
	So, $m$ is minimal (Example \ref{p:s-m-bi}.3) but not strongly minimal. Now apply  Corollary \ref{prop: knoco}.   
\end{proof}
	
The same idea works for every topological unital ring $M$ 
which is not a minimal topological ring.


%
%
%
%
%
%
%

\begin{theorem} \label{t:newStrMin} 
	Let $\F$ be a local field.  
	Then the multiplication map $m\colon \F \times \F \to \F$ is strongly minimal. 
\end{theorem}
\begin{proof} 
It is well known 
(see \cite[page 27]{MA})
 that there exists a non-trivial absolute value on $\F$ which generates the topology $\t$. 	
Assuming that $m$ is not strongly minimal, there exist coarser Hausdorff group topologies $\s_1, \t_1, \nu_1 \subseteq \t$
on $\F$  such that 
$$
m\colon (F,\s_1) \times (F,\t_1) \to (F,\nu_1)
$$
is continuous and $\s_1 \neq\ \t$ or $\t_1 \neq \t$. Then the corresponding Heisenberg group $(H(w),\g_1)= (\F \times \F) \rtimes \F$ is well defined with respect to the product topology $\g_1:=(\nu \times \s_1 \times \t_1)$ which is strictly coarser than the original topology  $\g:=(\nu \times \s \times \t)$.

Without restriction of generality, suppose that $\s_1 \neq \t$.  We first show that every $\s_1$-neighborhood of the zero element is unbounded with respect to the  absolute value.

Otherwise, there exists a bounded $\s_1$-neighborhood  $U$ of $0_\F$ and, clearly, we may assume that $U$ is $\t$-closed. By an  analogue of the Heine-Borel theorem for non-archimedean local fields (see \cite[Page 5]{V}), we deduce that $U$ is $\t$-compact. In particular, $\s_1$ and $\t$ coincide on $U.$ 
Since $U$ is a $\s_1$-neighborhood   of $0_\F$ it follows that $\s_1= \t,$ a contradiction.

Next, we show that for every unbounded subset $S \subset \F$ and every $\t$-neighborhood $V$ of $0_{\F}$ (or, even any nonempty $\t$-open subset) it holds that $1 \in SV$. Using the continuity of 
$$
m \colon (F,\s_1) \times (F,\t_1) \to (F,\nu_1),
$$ we obtain a contradiction to the fact that $\nu_1$ is Hausdorff.

Indeed, without restriction of generality,  let $V=B(0,\eps):=\{x \in \F: |x|<\eps\}$. 
Choose $s \in S$ such that $|s|> \frac{1}{\eps}$. 
Then $|s^{-1}|=|s|^{-1}<\eps$ and  $s^{-1} \in V$. Hence, $1=s \cdot s^{-1} \in SV.$  
\end{proof}

%

%
%
%

\section {The upper unitriangular group $\UT(n,K)$}
In the sequel  $K$ is a 	commutative  topological unital ring and $n\geq 2$ is a positive integer.  Let $G=\UT(n,K)$ be the upper unitriangular group over $K$ of degree $n$. 

Denote by $\UT^{m}(n,K)$ (where $m \in \{0,1,\cdots, n\}$) the subgroup of $\UT(n,K)$ having $m$ consecutive zero-diagonals parallel to the main diagonal. 
Hence, we have 
$$
\UT(n,K)=\UT^{0}(n,K) \geq \UT^{1}(n,K) \geq \cdots \UT^{n-1}(n,K)=\{I_n\}.
$$

\begin{remark} \label{r:n=3} 
For $n=3$ we have an important (and motivating) particular case of the Heisenberg group 
$$H(w)\simeq \UT(3,K):=
\Bigg\{\left( \begin{array}{ccc}
	1 &  a_{1,2} & a_{1,3} \\
	0 & 1 & a_{2,3}  \\
	0 & 0 & 1 
\end{array}\right)\bigg |  \ a_{i,j} \in K \Bigg\}.  
$$
Below we use several times that $\UT(3,K)$ is naturally isomorphic to the Hesienberg group $H(w) = (K \oplus K) \rtimes K$, modelled by the multiplication map $m \colon K \times K \to K$. 	
\end{remark}

Recall that $G=\UT(n,K)$ is nilpotent of class $n-1$. In fact, 
$$
G' = \UT^{1}(n,K)\cong \UT(n-1,K)
$$ 
as every matrix in the derived subgroup $G'$ has the first super diagonal consisting only of zero entries.  For $1\leq i<j\leq n,$ let $G_{i,j}$ be the 1-parameter subgroup of $G$ such that for every matrix $X\in G_{i,j}$ we have 
$p_{k,l}(X)=x_{k,l}=0$ if $k\neq l$ and $(k,l)\neq (i,j),$ 
 where $p_{k,l}\colon G \to K, \ p_{k,l}(X)=x_{k,l}$ is the canonical coordinate projection. 
 For $i<j$	and $x\in K$ consider also the \textit{transvection matrix} 
  $e_{i,j}(x):=I+x E_{i,j}\in  G_{i,j}$, where $E_{i,j}$ is the elementary matrix having $1$ on the $(i,j)$ entry and zeros elsewhere. 
%
%

\begin{remark}\label{rem:iso}
It is well known that the center 
$Z(G)$ 
 of $G=\UT(n,K)$ is
the corner 1-parameter group  
 $G_{1,n}= \UT^{n-2}(n,K)$ (see, in particular, \cite{KargMerz}).  Clearly, $Z(G)$ is an isomorphic copy of the additive group $K.$ In particular, $\UT(2,K)\cong K$ as  $\UT(2,K)$ is abelian.
\end{remark}

\begin{lemma}\label{lem:comandtrans}
Let  $M\ \in G$, $x\in K$ and $1\leq i<j<k\leq n.$ Then \ben 
\item $p_{i,k}([M,e_{j,k}(x)])=xp_{i,j}(M).$
\item   $p_{i,k}([e_{i,j}(x),M])=-xp_{j,k}(M^{-1}).$ 
\item $p_{n-2,n}([e_{n-2,n-1}(x),M])=xp_{n-1,n}(M).$
\een
\end{lemma}
\begin{proof}
(1) For  $1\leq i<j<k\leq n$ we have $$p_{i,k}([M,e_{j,k}(x)])=
p_{i,k}(M e_{j,k}(x) M^{-1} e_{j,k}(x)^{-1} )=\sum_{s=1}^n p_{i,s}(M e_{j,k}(x)) p_{s,k}(M^{-1} e_{j,k}(x)^{-1})=  $$$$=\sum_{s=1}^n\bigg(\sum_{\ell=1}^n p_{i,\ell}(M)p_{\ell,s}(e_{j,k}(x))
\bigg)\bigg(\sum_{t=1}^n p_{s,t}(M^{-1})p_{t,k}(e_{j,k}(x)^{-1})\bigg)=$$$$= 
\bigg(\sum_{s\in \{1,\ldots n\}\setminus \{k\}} p_{i,s}(M)
(p_{s,k}(M^{-1})-x p_{s,j}(M^{-1}))\bigg)+(p_{i,k}(M)+xp_{i,j}(M)) =$$$$=\bigg(\sum_{s=1}^n ( p_{i,s}(M)
)(p_{s,k}(M^{-1}))\bigg)-\bigg(x\sum_{s=1}^n ( p_{i,s}(M)
)(p_{s,j}(M^{-1}))\bigg)+(xp_{i,j}(M))=$$$$=p_{i,k}(I_n)-xp_{i,j}(I_n)+xp_{i,j}(M)=$$$$0+0+xp_{i,j}(M)=xp_{i,j}(M).$$
(2) We have $$p_{i,k}([e_{i,j}(x),M])=
p_{i,k}(e_{i,j}(x) Me_{i,j}(x)^{-1}M^{-1} )=\sum_{s=1}^n p_{i,s}(e_{i,j}(x) M) p_{s,k}(e_{i,j}(x)^{-1} M^{-1})=  $$$$=\sum_{s=1}^n\bigg(\sum_{\ell=1}^n p_{i,\ell}(e_{i,j}(x))p_{\ell,s}(M)
\bigg)\bigg(\sum_{t=1}^n p_{s,t}(e_{i,j}(x)^{-1})p_{t,k}(M^{-1})\bigg)=$$$$= 
\bigg(\sum_{s\in \{1,\ldots n\}\setminus \{i\}} (p_{i,s}(M)+xp_{j,s}(M))
(p_{s,k}(M^{-1})\bigg)+(p_{i,k}(M^{-1})-xp_{j,k}(M)^{-1}) =$$$$=\bigg(\sum_{s=1}^n ( p_{i,s}(M)
)(p_{s,k}(M^{-1}))\bigg)+\bigg(x\sum_{s=1}^n ( p_{j,s}(M)
)(p_{s,k}(M^{-1}))\bigg)-(xp_{j,k}(M^{-1}))=$$$$=p_{i,k}(I_n)+xp_{j,k}(I_n)-xp_{j,k}(M^{-1})=$$$$0+0-xp_{j,k}(M^{-1})=-xp_{j,k}(M^{-1}).$$ 
(3) By (2), $p_{n-2,n}([e_{n-2,n-1}(x),M])=-xp_{n-1,n}(M^{-1}).$ As $M$ is an upper unitriangular matrix, we have $p_{n-1,n}(M^{-1})=-p_{n-1,n}(M).$
So, we get $p_{n-2,n}([e_{n-2,n-1}(x),M])=xp_{n-1,n}(M).$
\end{proof}
\subsection{Key subgroups of the upper unitriangular group $\UT(n,K)$}

The following theorem is closely related to 
\cite[Proposition 2.9]{MEG95} and \cite[Lemma 5.9]{DM14} (take the particular case of $n=3$ and compare to Theorem \ref{t:AisKey} ). 

\begin{theorem} \label{t:1} 
	Let $K$ be a
	commutative topological unital ring and $n\geq 2$ be a positive integer. Then the center of $G=\UT(n,K)$ (that is, the corner subgroup $\UT^{n-2}(n,K)$) 
	is a key subgroup. 
\end{theorem}
\begin{proof}
	We proceed by induction on $n$. 
	For $n=2$ the assertion is trivial as $\UT(2,K)$ is abelian. The case of $n=3$ follows from  Theorem \ref{t:AisKey}  and  Example \ref{p:s-m-bi}.3, in view of the isomorphism $H(w)\simeq \UT(3,K)$, where $w \colon K\times K\to K$ is the multiplication map and $H(\omega)$ is the Hesienberg group.
	
	Let $\gamma_1 \subseteq \gamma$ be a coarser Hausdorff group topology on $G=\UT(n,K),$ where $n\geq 4,$ such that  $\gamma_1|_{Z(G)}=\gamma|_{Z(G)}$ and recall that
	$Z(G)=\UT^{n-2}(n,K)=G_{1,n}.$
	We have to prove that $\gamma_1=\gamma.$ By the induction hypothesis, $\gamma_1|_{G'}=\gamma|_{G'}$ as
	$$
	G' = \UT^{1}(n,K)\cong \UT(n-1,K). 
	$$ 
	In particular, $\g_1|_{G_{i,i+2}}=\g|_{G_{i,i+2}}$ for every $i\in\{1,\ldots, n-2\}$. We have to show that $\gamma_1=\gamma$. By Merson's Lemma, it is enough to show that 
	$\gamma_1/ {G'}=\gamma / {G'}$. 	To this aim, it suffices to show that all the projections   
	$$p_{i,i+1} \colon (G,\gamma_1) \to G_{i,i+1}, \ i\in \{1,\ldots n-1\} $$
	are continuous (where the  1-parameter subgroups $G_{i,i+1}$  carry the original topology $\g|_{G_{i,i+1}}=\g_1|_{G_{i,i+1}}$). 
	
	The idea of the multiplication map in the following claim comes from Lemma \ref{lem:comandtrans}.1. 
	\vskip 0.3cm
	\noindent
	\textbf{Claim 1}
	The multiplication map $$w \colon (G_{i,i+1}, \g_1 / \ker p_{i,i+1})\times (G_{i+1,i+2}, \g_1|_{G_{i+1,i+2}})  \to (G_{i,i+2}, \g_1|_{G_{i,i+2}}),$$ 
	$$  w(p_{i,i+1}(M),x)=x p_{i,i+1}(M), \ M\in G, \ x \in G_{i+1,i+2}$$ 
	\sk 
	is continuous for every $i\in \{1,\ldots, n-2\}.$
	\begin{proof}			
		Fix $i\in \{2,\ldots, n-1\}$ and $(p_{i,i+1}(M),x)\in  G_{i,i+1}\times G_{i+1,i+2}.$ 	 Let $O$ be a
		neighborhood of $xp_{i,i+1}(M)$ in $(G_{i,i+2}, \g_1|_{G_{i,i+2}}).$ By Lemma \ref{lem:comandtrans}.1, $$p_{i,i+2}([M,e_{i+1,i+2}(x)])=xp_{i,i+1}(M).$$  As $i\leq n-2$, it holds that $G_{i,i+2}\leq G'.$  Since  $\gamma_1|_{G'}=\gamma|_{G'}$  we can choose a neighborhood $W$ of $[M,e_{i+1,i+2}(x)]$ in $(G', \g|_{G'})$ such that $p_{i,i+2}(W) \subseteq  O,$ in view of the continuity of the projection   
		$$p_{i,i+2} \colon (G',\gamma|_{G'}) \to (G_{i,i+2},\gamma|_{G_{i,i+2}}).$$
		Using the $\g_1$-continuity of the commutator function on $G$  as well as the equality $\gamma_1|_{G'}=\gamma|_{G'}$ we can find  $\g_1$-neighborhoods  $U$ and $V$ of  $M$ and $e_{i+1,i+2}(x)$, respectively, such that $[A,B]\in W$ for every $A\in U, \ B\in V.$ In particular,  $$y p_{i,i+1}(N)=p_{i,i+2}([N,e_{i+1,i+2}(y)])\in  p_{i+1,i+2}(W) \subseteq  O$$  for every $N\in U, \ y\in V\cap G_{i+1,i+2}$.  This proves the continuity of $$w: (G_{i,i+1}, \g_1 / \ker p_{i,i+1})\times (G_{i+1,i+2}, \g_1|_{G_{i+1,i+2}})  \to (G_{i,i+2}, \g_1|_{G_{1,i+1}})$$ at the arbitrary pair $(p_{i,i+1}(M),x)\in  G_{i,i+1}\times G_{i+1,i+2}$. 
	\end{proof} 
	
	\noindent 
	\textbf{Claim 2} The multiplication map
	$$w:(G_{n-2,n-1}, \g_1|_{G_{n-2,n-1}}) \times (G_{n-1,n}, \g_1 / \ker p_{n-1,n})\to (G_{n-2,n}, \g_1|_{G_{n-2,n}}),$$$$  w(x,p_{n-1,n}(M))=x p_{n-1,n}(M), \ M\in G, \ x \in G_{n-2,n-1}$$
	is continuous.
	\begin{proof}
		Use the $\g_1$-continuity of the commutator function on $G$  as well as Lemma \ref{lem:comandtrans}.3.  On $G'$ the new topology is the same as the original topology. In particular, on $G_{n-2,n}.$
	\end{proof}	
	By Example \ref{p:s-m-bi}.3, the multiplication map $m:K \times K \to K$ is minimal. It follows from Claim 1 and Claim 2 that  $\g_1 / \ker p_{i,i+1}=\g / \ker p_{i,i+1}$ for every 	$i\in \{1,\ldots n-1\}$ and we complete the proof of the theorem. 
\end{proof}


\begin{proposition}\label{pro:admin}
	Let $K$ be a 	commutative  topological unital ring. Then the matrix group $\UT(n,K)$ is minimal  if and only if the additive group $(K,+)$ is minimal. 
\end{proposition}
\begin{proof} 
	The center $Z(G)$ of $G:=\UT(n,K)$ 
	is a key subgroup by  Theorem \ref{t:1} and $Z(G)$ is isomorphic to $K$ for every $n\geq 2$
by Remark \ref{rem:iso}. Now apply Lemma \ref{lem:ceniskey}. 	
\end{proof}

\begin{proposition} 
The corner integer subgroup	$\UT^{n-2}(n,\Z)$ 
(isomorphic to $\Z$) 
 is a key subgroup of $\UT(n,\R)$. 
\end{proposition}
\begin{proof}
	 Recall that $\Z \in \mathcal K (\R)$ (Proposition \ref{p:cocompact}.1). Hence, 
	  $\UT^{n-2}(n,\Z)$ is a key subgroup of $\UT^{n-2}(n,\R)$ 
	  (by the isomorphism). 
	  On the other hand, 
	by Theorem \ref{t:1}, $\UT^{n-2}(n,\R) \in \mathcal K(\UT(n,\R))$. 
	So, $\UT^{n-2}(n,\Z)$ is a key subgroup of $\UT(n,\R)$ by Lemma \ref{examples1}.2.  
\end{proof}


By a classical theorem of Prodanov--Stoyanov \cite{PS}, minimal abelian groups are precompact. In particular, if $K$ is a topological ring such that the additive group $(K,+)$ is minimal, so that $\UT(n,K)$ is minimal  in view of Proposition \ref{pro:admin}, then $K$ must be precompact. 
It is well known 
(see \cite{D83}, for example) 
that the topology of a precompact unital ring has a local base at $0$ consisting of open two-sided ideals. 
\begin{corollary}
	$\UT(n,(\Z, \tau_p))$ is a minimal topological group. 
\end{corollary}

	\begin{lemma} \label{l:also co-minimal}  \cite[Lemma 3.5.1]{DM10}   
	Let $H_2 \leq H_1 \leq G$. If $H_2$ is co-minimal in $G$, then 
	$H_1$ is also co-minimal in $G$.  
	\end{lemma}
 
\begin{lemma}\label{lem:chainofco}
	Let $G$ be a topological group with two normal subgroups $G_1$ and $G_2.$ If $G_2$ is a co-minimal  subgroup of $G_1$ and $G_1$ is co-minimal in $G,$ then $G_2$ is co-minimal in $G.$ 
\end{lemma}
\begin{proof}
	If $G$ is a topological group with two normal subgroups $G_1$ and $G_2$ such that $G_2\leq G_1$,  then $$G/G_2\cong  (G/G_1)/(G_1/G_2)$$ by the third isomorphism theorem  for topological groups (see \cite[Theorem 1.5.18]{AT} or  \cite[Theorem 3.2.8.b]{ADG}).
	Now use the co-minimality of $G_2$ in $G_1$ and the co-minimality of $G_1$ in $G.$
\end{proof}

\begin{example}
\label{ex:zinq}
As the rational torus 
$\Q/\Z$ is minimal (see \cite[Example 1.1.a]{DM14}) and $\Z$ is complete,  Proposition \ref{p:cocompact}.2 implies that $\Z$ is a key subgroup of $\Q.$ However,  $\Z$ is not  co-minimal in $\Q.$ Otherwise, as $\Q$ is co-minimal in $\R$ (being a dense subgroup), $\Z$ would be co-minimal in $\R$, in view of Lemma \ref{lem:chainofco}, but this contradicts Example  \ref{ex:R/Z}. 
\end{example}

The following theorem is closely related to results by Dikranjan and Megrelishvili  
\cite[Theorem 5.1]{DM10} (take the particular case of $n=3$ and compare to Theorem \ref{t:AisCo-Min} ).

%

\begin{theorem} \label{t:2} 
	Let $K$ be a
	commutative  topological unital ring. 
	Then  the following are equivalent:
	\ben \item The multiplication map $m \colon K\times K\to K$ is strongly minimal.
	\item 
	$Z(\UT(n,K))=\UT^{n-2}(n,K)$ is 
	co-minimal in $\UT(n,K)$ for every $n\geq 2.$
	\item $Z(\UT(3,K))$ is co-minimal in $\UT(3,K).$\een
\end{theorem}

\begin{proof}
$(1)\Rightarrow	(2)$ We prove 
this implication using induction on $n$. For $n=2$ the assertion is trivial as $\UT(2,K)$ is abelian.  The case of $n=3$ follows from 
	Theorem \ref{t:AisCo-Min} 
	and Remark \ref{r:n=3}. 
	
		For $G=\UT(n,K)$ with $n\geq 4$ we have $Z(G)\leq G''\leq G'.$ So, the induction hypothesis and 
	Lemma \ref{l:also co-minimal} 
	imply that both $G_{1,n}=Z(G)$ and $G''$ are co-minimal in $G'.$ 
	So by Lemma \ref{lem:chainofco} it suffices to show that $G'$ is co-minimal in $G.$ 
	Let $\gamma_1 \in \mathcal{T}_{\downarrow}(G,\g)$. We need to show that $\gamma_1/ {G'}=\gamma / {G'}$. Now we follow the arguments appearing in the proof of Theorem \ref{t:1} including Claim 1 and Claim 2 with some modifications. 
	This time we do not  have the assumption that
	$\gamma_1|_{G'}=\gamma|_{G'}$ and so we do not know that
	$\g_1|_{G_{i,i+2}}=\g|_{G_{i,i+2}}$ for every $i\in\{1,\ldots, n-2\}$. Nevertheless, since the multiplication map is strongly minimal it is enough to assume that $\g_1|_{G_{i,i+2}}\subseteq \g|_{G_{i,i+2}}.$ Note also that the projections
		$$p_{i,i+2} \colon (G',\gamma_1|_{G'}) \to (G_{i,i+2},\gamma|_{G_{i,i+2}})$$  are continuous for every $i\in\{1,\ldots, n-2\},$ in view of the co-minimality of $G''$ in $G'.$ 
		
		$(2)\Rightarrow	(3)$ Trivial.
		
		$(3)\Rightarrow	(1)$ Use Proposition \ref{t:AisCo-Min}.
\end{proof}
%
Recall that if $m$ is strongly minimal, then $K$ is necessarily a minimal topological ring (see Example \ref{p:s-m-bi}.3). 
\begin{proposition} \label{p:oneMore}  
	 $\UT^{n-2}(n, \Z)$ is not co-minimal in $\UT(n,\Z)$
 being, however, a key subgroup.  
\end{proposition}
\begin{proof} By Theorem \ref{t:1} 
	the corner subgroup
	$\UT^{n-2}(n,\Z)=Z(\UT(n,K))$ is a key subgroup in $\UT(n,\Z)$.  
	On the other hand, $\UT^{n-2}(n, \Z)$ (which is isomorphic to $\Z$) is not co-minimal in $\UT(n,\Z)$ by Theorem \ref{t:2} (for $K:=\Z$). Indeed, observe that the multiplication map $m \colon \Z \times \Z \to \Z$ is not strongly minimal  (compare with Example \ref{p:s-m-bi}.3) because $\Z$ is not a minimal topological ring. 
\end{proof}

\begin{theorem} \label{t:LocField} 
	Let $\F$ be a local field or an archimedean absolute valued field. 
	 Then the center $Z(\UT(n,\F))$ (the corner subgroup $\UT^{n-2}(n, \F)$) 
	is co-minimal in $\UT(n,\F).$
\end{theorem}
\begin{proof}
	Theorems \ref{t:2} and \ref{t:newStrMin} yield this  result for local fields. For archimedean absolute valued field use also Example  \ref{p:s-m-bi}.4 which asserts that the multiplication map $m\colon \F \times \F \to \F$ is strongly minimal.
\end{proof}

\sk  
\subsection{Co-key subgroups in $\UT(n,K)$} 
The next result and its proof 
 was motivated by  \cite[Theorem 3.8]{SH}.
\begin{theorem} \label{t:co-1} 
	Let $K$ be a commutative topological unital ring and $n\geq 3$ be a positive integer. Then every non-corner   1-parameter subgroup $G^n_{i,j}$ 
	 of $G=\UT(n,K)$ (i.e., when $(i,j)\neq (1,n)$) is a co-key subgroup.
\end{theorem}
\begin{proof} 
  Assume first that $i=1$ or $j=n$. By Example \ref{p:s-m-bi}.3,  the biadditive map  $$w_n\colon K^n\times K^n\to K, \ w_n(\bar{x},\bar{y})= \sum_{i=1}^n x_iy_i$$
   is minimal. By Theorem \ref{t:AisKey}, both subgroups $$\left( \begin{array}{ccc}
  	1 & K^n & 0 \\
  	0 & 1 & 0 \\
  	0 & 0 & 1 
  \end{array}\right), \left( \begin{array}{ccc}
  1 & 0 & 0 \\
  0 & 1 & K^n \\
  0 & 0 & 1 
\end{array}\right) $$ are co-key in $H(w_n).$ Since $G^n_{i,j}$ is a subgroup of one of these two subgroups and  $H(w_n)$ is a subgroup of $G$ it follows from Lemma \ref{l:LargerSubAlsoCoKey}  that $G^n_{i,j}\in \mathcal{CK}(G).$ 

Now suppose that $i>1$ and $j<n.$  Denote by $\widetilde{\UT}(n,K)$ the subgroup of $G$ consisting of all matrices having zeros outside the diagonal at the first $i-1$ rows. Deleting the first $i-1$ rows and $i-1$ columns yields a topological group isomorphism $f\colon \widetilde{\UT}(n,K)\to \UT(n+1-i,K)$ such that $f(G^n_{i,j})=G^{n+1-i}_{1,j+1-i}.$ By the previous case, $G^{n+1-i}_{1,j+1-i}$  is co-key in $\UT(n+1-i,K).$ Since $f$ is an isomorphism it follows that $G^n_{i,j}$ is co-key in $\widetilde{\UT}(n,K).$ As $\widetilde{\UT}(n,K)\leq G$ and using Lemma  \ref{l:LargerSubAlsoCoKey} we deduce   that $G^n_{i,j}\in \mathcal{CK}(G).$ 
\end{proof}


Using similar arguments  and  	Theorem \ref{t:AisCo-Min} instead of Theorem \ref{t:AisKey} one can prove the following.

\begin{theorem}
	Let $K$ be a commutative topological unital ring and $n\geq 3$ be a positive integer such that the biadditive map 
	$$w_n\colon K^n\times K^n\to K, \ w_n(\bar{x},\bar{y})= \sum_{i=1}^n x_iy_i$$ 
	is strongly minimal. Then the 1-parameter subgroup 
	$G_{i,j}$ 
	is  relatively minimal in
	$G=\UT(n,K)$  whenever $(i,j)\neq (1,n).$
\end{theorem}

\subsection{Some comments}

There are several good reasons to study central key subgroups
and central co-minimal subgroups.  
Let us say that a topological group $G$ is \textit{center-key} if its center $Z(G)$ is a key subgroup in $G$. 
According to Theorem \ref{t:1}, $\UT(n,K)$ is center-key
 for every commutative topological unital ring $K$. 
 The general linear group $\GL(n,\F)$ is center-key for every local field $\F$ 
 (Theorem \ref{t:projective}). 

 
\begin{question} \label{q:center-key} 
	Study natural topological groups which are center-key.  
\end{question}

It is well known (see \cite[Proposition 7.2.5]{DPS89}) that a closed central subgroup of a minimal group is minimal. Taking into account also Lemma \ref{examples1}.1 we obtain that 
a topological group $G$ is minimal if and only if $Z(G)$ is a minimal key subgroup. 
In case that $G$ is Raikov complete  
$G$ is minimal if and only if $Z(G)$ is compact in view of Prodanov-Stoyanov theorem \cite{PS} mentioned above. 

In \cite{BadLeib}, the authors show that groups from a large class of algebraic groups are minimal if and only if the center is compact. Such groups, of course, are center-key.   

It should be interesting to find additional examples of 
topological groups which are center-key (where the center is not necessarily compact). In particular, which of them are co-minimal. This is especially interesting in view of 
the injectivity properties of co-minimal subgroups. See Section \ref{s:inj}. 

  
\sk 

\section {Injectivity property and co-minimal subgroups} 
\label{s:inj} 
%

For a topological group $(G,\g)$ denote by $\mathcal{T}_{\downarrow}(G,\g)$ (or, simply $\mathcal{T}_{\downarrow}(G)$) the poset (in fact, sup-semilattice) of all Hausdorff group topologies $\t$ on $G$ such that $\t \subseteq \g$. 
For every pair $\s, \t \in \mathcal{T}_{\downarrow}(G,\g)$ we have $\sup\{\s,\t\} \in \mathcal{T}_{\downarrow}(G,\g)$. 
The following family 
\begin{equation} \label{eq:sup-top} 
	\{U \cap V: \  U \in \s(e), V \in \t(e)\}
\end{equation} 
constitutes a local base of $\sup\{\s,\t\}$ at the neutral element $e \in G$, where $\s(e)$ and $\t(e)$ are the sets of all $\s$-neighborhoods and $\t$-neighborhoods of $e$.

If $H$ is a subgroup of $G$, then we have a natural order preserving restriction map 
$$r_H \colon \mathcal{T}_{\downarrow}(G) \to \mathcal{T}_{\downarrow}(H), \ \s \mapsto \s|_H.$$  
Furthermore, it is a morphism of sup-semilattices because   
for every pair $\s,\t \in \mathcal{T}_{\downarrow}(G)$ we have (using  Equation \ref{eq:sup-top}) 
$$
\sup\{r_H(\s),r_H(\t)\}=r_H(\sup\{s,\t\}). 
$$
In general, $r_H$ need not be injective even when $H$ is a  key subgroup (see Remark  \ref{r:super-rem}.  
It is onto 
if $H$ is a central subgroup (see Remark  
\ref{r:ONTO}).

\begin{definition} \label{d:inj-key} 
	Let $H$ be a subgroup of a topological group $(G,\gamma)$. 
	We say that $H$ is 
	an \textbf{inj-key subgroup} of $G$ if for every pair 
	$\gamma_1,\gamma_2 \in \mathcal{T}_{\downarrow}(G)$ on $G$ the coincidence of the two topologies on $H$ (that is,  $\gamma_1|_H = \gamma_2|_H$) 
	implies that  $\gamma_1 = \gamma_2$. Equivalently, this means that $$r_H \colon \mathcal{T}_{\downarrow}(G) \to \mathcal{T}_{\downarrow}(H), \ \s \mapsto \s|_H$$ is \textbf{injective}.  
\end{definition}
Clearly, inj-key is key. Below we show that co-minimal implies inj-key. So, we have 

\centerline{\large{co-minimal $\Rightarrow$ inj-key $\Rightarrow$ key}} 
\noindent If $H$ is a central subgroup, then 
(by Lemma \ref{l:ExtLem}.1) $r_H$ is onto and so 
the injectivity of $r_H$ can be replaced by bijectivity and we we get an isomorphism of semilattices.  

Moreover, in this case $H$ is inj-key if and only if it is co-minimal (Theorem \ref{p:inj-key=co-min for central}).
It is an open question if the assumption on the centrality of $H$ can be removed.  

\sk 
Recall  that
every closed central subgroup of a minimal group is minimal 
\cite[Proposition 7.2.5]{DPS89}.
A simple proof comes from a natural  extension argument (see Extension Lemma \ref{l:ExtLem}.1 below) used by Prodanov (for abelian groups) \cite{Prodanov77}.
Different variations of this 
result can be found in several papers. See, for example, 
\cite[Proposition 4.7]{AHSTW06}, \cite[Lemma 3.17]{HPTX19}  
and \cite[Theorem 6.5]{DM10}. 

\begin{lemma} \label{l:ExtLem} (Extension Lemma) 
	Let $H$ be a central subgroup of $(G,\g)$. Let $\s$ be a coarser Hausdorff group topology on $H$. 
	\begin{enumerate}
		\item 
		There exists a Hausdorff group topology $\s^*  \subset \g$ on $G$ such that $\s^* |_H=\s$. The corresponding local base is generated by the family 
		\begin{equation} \label{eq:LocBase} 
			\{UV: 
			U \in \g(e), V \in \s(e)\}. 
		\end{equation}
		\item $\s^*/H=\g/H$. 
		\item  
		If $H$ is 
		$\g$-closed,
		then 
		it is $\s^*$-closed. 
	\end{enumerate} 
\end{lemma} 
	\begin{proof}
		(1) For a direct verification see, 
		\cite[Proposition 3]{Prodanov77} and 
		\cite[Theorem 6.5]{DM10} (where $H$ is not necessarily closed). 
		
		(2) Easily follows from (1) because if $q \colon G \to G/H, \ g \mapsto gH$ is the natural projection, 
		then $q(UV)=q(UH)=q(U)$ (see Equation \ref{eq:LocBase}) for every $U \in \g(e), V \in \s(e)$.

		(3) If $H$ is $\g$-closed, then $\g/H$ is Hausdorff. Hence, by (2), $\s^*/H$ is also Hausdorff. So, $H$ is $\s^*$-closed.  
	\end{proof}

\begin{remark} \label{r:ONTO} \
	\begin{enumerate}
		\item $r_H$ is onto for every central subgroup $H$ of $G$.  
		
		\sk 
		It is a consequence of Lemma \ref{l:ExtLem}.1. 
		
		\item $r_H$ need not be onto (even for abelian subgroups of index 2).  
		
		\sk 
		
		There exists a group $G$ with a subgroup $H$ such that not every topology on $H$ can be extended to a topology of $G$. See, for example, \cite[Exercise 2, page 131]{RD}. 
		This issue 
		was carefully discussed in Section 4.4 of the book \cite{ADG} (based on results of \cite{DS07}).  
		More precisely, as in \cite[Example 4.4.8]{ADG}, take any discontinuous involution $f \colon \T \to \T$ of the circle group $\T$. 
		Consider the discrete semidirect product $G:= \T_d \rtimes <f>$. Then for the 2-index subgroup $H:=\T_d$ of $G$ there is no Hausdorff group topology on $G$ which extends the usual topology of $\T$.    	
	\end{enumerate}  
\end{remark}

\sk 
A 
closed subgroup $H$ of $(G,\g)$ is said to be \textit{strongly closed} (in the sense of \cite{DM10}) if $H$ remains closed in $(G,\t)$ for every  $\t \in \mathcal{T}_{\downarrow}(G,\g)$. 
For discrete groups this is just the classical  \textit{unconditionally closed} subgroups.  
Similarly, an open subgroup $H$ of $(G,\g)$ is called  \textit{strongly open}
if it is open in $(G,\t)$ for every  $\t \in \mathcal{T}_{\downarrow}(G,\g)$.
\begin{theorem} \label{p:inj-key=co-min for central} \ 
	\begin{enumerate}
		\item Every co-minimal subgroup is inj-key.  
		\item Let $H$ be a central subgroup of $(G,\g)$. Then $H$ is inj-key if and only if $H$ is co-minimal in $G$. 	 
		\item Let $H$ be a strongly closed subgroup in $G$  which is co-compact and Raikov complete. Then $H$ is co-minimal. 
		\item Let $H$ be a  subgroup of $(G,\g).$ Then $H$ is inj-key if and only if $H$ is a key subgroup of $(G,\eta)$ for every $\eta\in \mathcal{T}_{\downarrow}(G,\g).$
	\end{enumerate}	
\end{theorem}
\begin{proof} 
	(1) Let $H$ be co-minimal in $(G,\g)$. 
	Assume that $\t, \s \in \mathcal{T}_{\downarrow}(G,\g)$ such that $\t|_H=\s|_H$. 
	We have to show that $\s=\t$. 
	We cannot apply Merson's Lemma directly to the pair $\t, \s$ because $\t$ and $\s$ are not comparable, in general. 
	
	Let $\eta :=\sup\{\t,\s\} \in \mathcal{T}_{\downarrow}(G,\g)$. 
	According to the simple description of $\eta(e)$ (see Equation \ref{eq:sup-top}), $\t|_H=\s|_H$ implies   
	that 
	$$
	\eta|_H=\t|_H=\s|_H.
	$$
	Since $H$ is co-minimal in $(G,\g)$ we have $\t/H=\g/H$. Similarly, $\s/H=\g/H$. 
	Since $\eta :=\sup\{\t,\s\}$ we have $\t/H \subseteq \eta /H$. On the other hand, $\eta \subseteq \g$. Hence, $\eta /H \subseteq \t /H$. Therefore, $\eta /H = \t /H$. Since $\t \subseteq \eta$, one may apply Merson's Lemma which establishes that $\t=\eta$. Similarly, $\s=\eta$. So, we can conclude that $\t=\s$.   
	
	(2)
	Assume that $H$ is inj-key. We have to show that $H$ is co-minimal. If not, then there exists a $T_2$
	group topology $\t \subset \g$ on $G$ such that $\t/H \subsetneq \g/H$. Then clearly, $\t \neq \g$. 
	Then, $\s:=\t|_H \subsetneq \g|_H$, because, otherwise $H$ is not inj-key. By Extension Lemma \ref{l:ExtLem}, we have a $T_2$ group topology $\s^*$ 
	on $G$ such that 
	$\s^*|_H=\s$ and $\s^*/H=\g/H$. By the latter equality, $\s^* \neq \t$. Now the equality $\s^*|_H=\s=\t|_H$ implies that 
	$r_H \colon \mathcal{T}_{\downarrow}(G) \to \mathcal{T}_{\downarrow}(H)$ is not injective. Meaning that  
	$H$ is not inj-key in $(G,\g)$.  
	
	(3) 	Assume that $\t \in \mathcal{T}_{\downarrow}(G,\g)$. Then $H$ is $\t$-closed. Therefore, $\t/H$ is a Hausdorff group topology. Since $\t/H \subseteq \g/H$ and $\g/H$ is compact, we get 
	$\t/H = \g/H$.  
	
	(4) The necessity of the condition is trivial. Let us  prove that it is also sufficient. To this aim let   $\t, \s \in \mathcal{T}_{\downarrow}(G,\g)$ such that $\t|_H=\s|_H$. 	We have to show that $\s=\t$.   Using the same arguments as  in  the proof of  (1), we obtain   that $
	\eta|_H=\t|_H=\s|_H,
	$	where $\eta :=\sup\{\t,\s\} \in \mathcal{T}_{\downarrow}(G,\g)$. By our assumption, $H$ is a key subgroup of $(G,\eta)$. It follows that $\eta=\s=\t.$
\end{proof}
\begin{proposition}
Let $H$ be an open normal subgroup of a topological group $(G,\g).$ The following conditions are equivalent:
\ben
\item $H$ is co-minimal in $G.$
\item $H$ is inj-key in $G.$
\item $H$ is strongly open in  $(G,\g).$ 
\een
\end{proposition}
\begin{proof}
(1) $\Rightarrow$ (2): Use Theorem \ref{p:inj-key=co-min for central}.1.\\
(3)  $\Rightarrow$ (1): See item (b) of \cite[Theorem 3.10]{DM10}.\\
(2)  $\Rightarrow$ (3): Let us prove that if $H$ is not strongly open, then it is not inj-key in $(G,\g).$  By definition there exists $\t \in \mathcal{T}_{\downarrow}(G,\g)$ such that $H$ is not $\t$-open. Let $\s$ be the coarsest (not necessarily Hausdorff) group topology on $G$ making $H$ open in $G.$  It is easy to see that as a local base one can take  $\s(e)=\{H\}$ as $H$ is normal. Now, let $\eta :=\sup\{\t,\s\} \in \mathcal{T}_{\downarrow}(G,\g)$. Then $\eta\neq \t$ as $H$ is $\eta$-open. However, $\t|_H=\eta|_H$. Indeed, by Equation \ref{eq:sup-top} we have  $\eta(e)=\{H\cap V: \  V\in \t(e)\}$.  The latter is also local base at the identity for the subspace topologies 
$(H,\eta|_H)$ and $(H,\t|_H).$ So, $\t|_H=\eta|_H$ and  $H$ is not inj-key in $(G,\g).$
\end{proof}
\begin{remark}
The equivalence of conditions (1) and (3) holds true also when $H$ is not normal. If there exists an open inj-key subgroup $H$ that is not strongly open in $(G,\g)$, then this subgroup will  not be co-minimal.
\end{remark}
We list some more properties of inj-key subgroups.
\begin{lemma} \label{examples2} Let $H_2 \leq H_1$ be subgroups of a topological group $(G,\g).$  
	\ben  

	\item 	
If	$H_2$ is inj-key in
	$G$, then $H_1$ is  inj-key in
	$G.$
	\item 	If $H_2$ is  inj-key in $H_1$  and $H_1$ is  inj-key in $G,$ then $H_2$ is inj-key in $G.$
	\item  $H_2$ is  inj-key in $G$ if and only if $cl(H_2)$ is inj-key in $G.$
	\een
\end{lemma}
\begin{proof}
(1) Trivial.

(2) In view of Theorem \ref{p:inj-key=co-min for central}.4, it is equivalent to show that $H_2$ is a key subgroup of  $(G,\eta)$ for every $\eta\in \mathcal{T}_{\downarrow}(G,\g).$ By our assumption, $H_2$ is a key subgroup of 
$(H_1, \eta|_{H_1})$ and $H_1$ is is a key subgroup of  $(G,\eta).$ Now, use Lemma \ref{examples1}.1.

(3) Let $H_1=cl(H_2).$ Assume first that $H_2$ is inj-key in $G.$  By (1),  $H_1$ is  inj-key in
$G.$ Now, suppose that $H_1$ is  inj-key in
$G.$ As $H_2$ is dense in $H_1$ it must be  co-minimal and, in particular, inj-key in $H_1$ in view of  Theorem \ref{p:inj-key=co-min for central}.1.  By (2), we deduce that
 $H_2$ is  inj-key in
$G.$
\end{proof}

\begin{fact} \label{f:strCl-DM} \ 
	\begin{enumerate}
		\item \cite[Theorem 3.10]{DM10} 	Let $H$ be a closed subgroup in $G$. 
		\begin{enumerate}
			\item If $H$ is a co-minimal subgroup, then it is strongly closed. 
			\item If $H$ is a strongly closed normal subgroup of $G$ and the factor group $G/H$ is minimal, then $H$ is co-minimal.
		\end{enumerate}
		\item \cite[Corollary 6.6]{DM10} 
		Let $H$ be a closed central subgroup of $G$. Then 
		$H$ is relatively minimal in $G$ if and only if $H$ is 
		minimal.
	\end{enumerate}
\end{fact}

Note that Fact \ref{f:strCl-DM}.2 easily follows from Lemma \ref{l:ExtLem}.1. 

\begin{theorem} \label{t:projective} 
	The center $Z$ (the subgroup of all invertible scalar matrices) of the general linear group $\GL(n,\F)$ is 
	co-minimal 
	for every local field $\F$.  
\end{theorem}
\begin{proof}
	By \cite{MS-Fermat}, the projective group $\PGL(n,\F):=\GL(n,\F) / Z$ is totally minimal. 
	Since 
	$Z$ is unconditionally closed in 
	$\GL(n,\F)$,  
	we may apply 
	Fact \ref{f:strCl-DM}.1(b).  
\end{proof}


\begin{corollary} \label{c:co-minCentral is S-closed} 
	Every central closed inj-key subgroup $H$ of $G$ is strongly closed.  
\end{corollary}
\begin{proof}
	Combine Theorem \ref{p:inj-key=co-min for central}.2 and Fact \ref{f:strCl-DM}.1(a). 
	
	\textit{We provide also an additional (constructive) proof}: 
	Let $\t \in \mathcal{T}_{\downarrow}(G,\g)$ such that $H$ is not $\t$-closed. For $\g$ and $\s:=\t|_H$ consider the canonically defined $\s^* \in \mathcal{T}_{\downarrow}(G,\g)$ (Extension Lemma \ref{l:ExtLem}). Then $\s^*|_H=\t|_H$ and $H$ is $\s^*$-closed. The latter condition implies that $\s^* \neq \t$.   
\end{proof}

Recall 
that a subgroup $H$ of a topological group $(G,\g)$ is said to be \textit{potentially dense} (coming back from A.A. Markov) if there exists $\g_1 \in \mathcal{T}_{\downarrow}(G,\g)$ such that $H$ is $\g_1$-dense.


\begin{remark} \label{r:PotDense}  
	Clearly, any proper potentially dense 
	closed (central) subgroup is not strongly closed, hence also not co-minimal (resp., not inj-key)  by
	 Fact \ref{f:strCl-DM}.1(a) 
	(and Corollary \ref{c:co-minCentral is S-closed}). 
	A very typical example of a potentially dense subgroup is $\Z$ in $\R$. There exists an injective continuous homomorphism 
	$$f \colon \R \to \T^2, \ \ f(t)=(e^{it}, e^{\sqrt{2}it})$$ 
	such that $f(\Z)$ is dense in $\T^2$, hence, also in $f(\R)$.   
	Similar arguments work also for $\Z^n \subset \R^n$, 
	or 
	non-trivial proper 
	$k\Z \subsetneq\Z$. 
\end{remark}

%

\sk 
Recall that the \textit{Bohr topology} $\g_b$ on an  
abelian group $(G,\g)$ coincides with the weak topology generated by all continuous characters $f \colon G \to \T$. 
This is also the finest precompact topology on $G$ coarser than $\g$.
If $G$ is a locally compact abelian (LCA) group, then 
the characters separate the points. 
So, $\g_b$ is Hausdorff. This assignment is functorial. More precisely, every  continuous homomorphism between two LCA groups $h \colon (G_1,\g_1) \to (G_2,\g_2)$ implies the continuity $h \colon (G_1,(\g_1)_b) \to (G_2,(\g_2)_b)$. In fact, the converse is also true. See \cite[Theorem 1.2]{Trig}.  

Every continuous character on a closed subgroup $H$ of a LCA group $G$ can be extended to a continuous character defined on $G$, \cite[Prop. 13.5.5]{ADG}. This implies that $H$ is also $\g_b$ closed in $G$. 

Recall also the following useful known lemma (see for example, \cite[Cor. 11.2.8]{ADG}) for discrete $G$ but actually with the same proof). 
%
\begin{lemma} \label{l:factBohr} 
	$\g_b/H=(\g/H)_b$ holds for every LCA group $(G,\g)$ and its closed subgroup $H$. 	
\end{lemma} 
\begin{proof} 
	Since $\g_b$ is precompact, then also $\g_b/H$ is precompact. By the universality property of the Bohr topology $(\g/H)_b$ on the group $G/H$ 
	we have $\g_b/H \subseteq (\g/H)_b$. As we already mentioned, the continuity of $q \colon (G,\g) \to (G/H,\g/H)$ implies the continuity of the same map with the Bohr topologies $q \colon (G,\g_b) \to (G/H,(\g/H)_b)$. Now, consider the quotient map $q \colon (G,\g_b) \to (G/H, \g_b /H)$. By the definition of the quotient topology 
	we have 
	$(\g/H)_b \subseteq \g_b/H.$ Hence, $\g_b/H = (\g/H)_b$. 
\end{proof}

\begin{theorem} \label{t:inj-key is co-comp in LCA} 
	Let $G$ be a locally compact abelian group and $H$ be its closed  subgroup. 
	\begin{enumerate}
		\item If $H$ is co-minimal, then it is co-compact. 
		\item $H$ is co-minimal 
		if and only if $H$ is strongly closed and co-compact.   
		\item If $G$ is a discrete abelian group, then its subgroup $H$ is co-minimal if and only if $H$ is unconditionally closed in $G$ with 
		finite index. 
	\end{enumerate}
\end{theorem}
\begin{proof} (1) 
	Let $\g$ be the topology of $G$ and $\g_b$ the Bohr topology (which is always Hausdorff for every LCA group $G$). 
	Let $q \colon G \to G/H$ be the factor map and $q^{-1}(\g /H)$ be the preimage topology on $G$. Define $\s:=\sup\{\g_b, q^{-1}(\g /H)\}$. Then $\s$ is a Hausdorff group topology on $G$ and $\g_b\leq \s \leq \g$. Furthermore, $\s|_H=\g_b|_H$.  Since $H$ is inj-key, 
	we necessarily have $\g_b=\s$. Therefore, $\s/H=\g_b/H$. 
	
	On the other hand, $\s /H \supseteq q^{-1}(\g /H) /H =\g/H$. Since $\s \leq \g$, we get $\s /H =\g/H$. So, $\g /H =\g_b /H$. 
	Lemma \ref{l:factBohr}  guarantees that $\g_b/H=(\g/H)_b$. 
	So, the factor group topology on the locally compact abelian group $(G/H, \g /H)$ coincides with its Bohr topology. The latter is always precompact. Therefore, $G/H$ must be compact.  
	
	(2) If $H$ is co-minimal in $G$, then $H$ is co-compact by (1) and strongly closed by Fact \ref{f:strCl-DM}.1(a). 
	Conversely, if $H$ is strongly closed and co-compact, 
	then we can apply 
	 Fact \ref{f:strCl-DM}.1(b) 
	 or Theorem \ref{p:inj-key=co-min for central}.3. 
	
	(3) Easily follows from (2) as a particular case. 
\end{proof}

\begin{example} \label{e:co-minInLCA} 
	There exists a discrete abelian countable infinite (non-minimal) group $G$ containing a proper co-minimal subgroup $H$. 
\end{example}
\begin{proof}
	Let $p,q$ be distinct primes, 
	$P:=\Z/p\Z$ and $Q:=\Z/q\Z$ are cyclic groups with $p$ and $q$ elements, respectively. Consider the direct sum $H:=P^{(\omega)}$ of countably many copies of $P$. Then the discrete group $G:=Q \oplus H$ is the desired countable group by Theorem  \ref{t:inj-key is co-comp in LCA}.3. Indeed, $H$ is unconditionally closed in $G$ because $H=\{g \in G: \  pg=0\}$ is an elementary algebraic set in $G$.  
\end{proof}


\begin{remark} \label{r:super-rem} 
	A co-compact key subgroup need not be inj-key. 
	
	(a) $\Z^n$ is key in $\R^n$ but not inj-key. 
	
	(b) Any 
	proper non-trivial subgroup $k\Z$ of $\Z$ 
	is key but not inj-key. 
	
	\sk 		
	These subgroups are not inj-key by Remark \ref{r:PotDense}. 
	In both cases the subgroups are key subgroups being locally compact and co-compact (Proposition \ref{p:cocompact}.1).

	\end{remark}

	\begin{example} \label{ex:inj-keyCorHeis}  \ Co-minimal subgroups are inj-key by Theorem \ref{p:inj-key=co-min for central}.1. As a corollary we mention here some remarkable cases where the restriction map 
$r_A \colon \mathcal{T}_{\downarrow}(G) \to \mathcal{T}_{\downarrow}(H)$ is bijective (isomorphism of semilattices) for 
the center $H:=Z(G)$. 

\begin{enumerate}
	\item $G:=H(w)$ is the Heisenberg group of a strongly minimal biadditive mapping $w \colon E \times F \to A$.  
	In particular this is true in the following cases:
	\begin{enumerate}
		\item For every local field $\F$ and the map $w \colon \F \times \F \to \F$. 
		
		This observation  seems to be new even for the classical 3-dimensional Heisenberg group.
		\item 	
		For every Banach space $V$ and the canonical bilinear map $w \colon V \times V^* \to \R$. 
		
	\end{enumerate}
	
	\item $G:=\UT (n,\F)$ be the matrix group of all unitriangular matrices where $\F$ be a local field or an archimedean absolute valued field.   Theorem \ref{t:LocField} 
	implies that (the restriction on its center) 
	$r_{\F} \colon \mathcal{T}_{\downarrow}(\UT (n,\F)) \to \mathcal{T}_{\downarrow}(\F)$ 
	is an isomorphism.

	\item Let $\GL(n,\F)$ be the general linear group defined over local field $\F$. 
	By Theorem \ref{t:projective}, the center $Z$ (which is isomorphic to $\F^{\times}:=\F\setminus \{0\}$) is co-minimal in $\GL(n,\F)$. Hence, (the restriction on its center) $r_{\F^{\times}} \colon \mathcal{T}_{\downarrow}(\GL(n,\F)) \to \mathcal{T}_{\downarrow}(\F^{\times})$ is an isomorphism. 
\end{enumerate}
\end{example}


\sk 
\subsection{More questions} 

\begin{question} \label{q:key} 
Study (characterize) closed subgroups $H$ in \textbf{locally compact abelian} groups $G$ which are key. 
\end{question} 

\begin{remark} \label{r:submaximal} 
Using results of Dikranjan-Tkachenko-Yaschenko \cite[Section 4]{DTY} one may characterize key subgroups of \textbf{discrete abelian} groups. Namely, 
a subgroup $H$ of a discrete abelian group $G$ is a key subgroup if and only if $H$ is open in the \textit{submaximal topology} $M_G$.  
Note that key subgroups in discrete groups are exactly subgroups which are \textbf{not} \textit{potentially discrete} subgroups in terms of \cite{DTY}.  


\end{remark}


\begin{prob} \label{q:main}  \ 
	Find examples of inj-key subgroups which are not  co-minimal. 
\end{prob}

As we know by Theorem \ref{p:inj-key=co-min for central}.2, every central inj-key subgroup is co-minimal. In particular, it is always the case in abelian topological groups. So it makes sense to examine subgroups of massive non-commutative groups.  
Below in Subsection \ref{s:counterexample} we give a concrete counterexample which resolves Problem \ref{q:main}. However, it seems to be an attractive question to find more 
of  such examples.  
In particular, what about the closed embeddings $i \colon H \hookrightarrow G$ which are epimorphisms (in the category of Hausdorff topological groups)? 


\begin{proposition}
Let $H$ be a proper co-compact  Raikov complete subgroup of a topological group $(G,\g).$ If $G$ is non-minimal  and $H$ is    $\s$-dense in any strictly coarser group topology $\sigma\subseteq \gamma$,  then $H$ is inj-key but not co-minimal.
\end{proposition}
\begin{proof}
Let $\sigma\subseteq \gamma$ be a strictly coarser group topology on $G.$ Note that such a group topology exists since $G$ is not minimal. By our assumption, $\g/H\neq \s/H$ as $\g/H$ is Hausdorff while $\s/H$ is the trivial topology on a non-trivial group. This proves that $H$ is not co-minimal. 

To prove that $H$ is inj-key in $(G,\g)$ it is equivalent to show by Theorem \ref{p:inj-key=co-min for central}.4  that $H$ is a key subgroup of $(G,\eta)$ for every $\eta\in \mathcal{T}_{\downarrow}(G,\g).$ If $\eta=\g$, then we may use  Proposition \ref{p:cocompact}.1. If $\eta$ is strictly coarser than $\g$, then by our assumption $H$ is $\eta$-dense in $G.$ This implies that $H$ is   a key (in fact, even co-minimal) subgroup of $(G,\eta).$ 
\end{proof}

\sk 
\subsection{Inj-key subgroup which is not co-minimal}
\label{s:counterexample} 


V. Pestov proposed to examine Polish groups $G$ with \textit{metrizable} universal minimal $G$-flow $M(G)$. 
We consider here a particular case of such Polish $G$ which was analyzed in 
 \cite{GM-UltraHom}.

Let $\Q_0:=\Q/\Z$ be the circled rationals with its discrete topology. Consider the group $G:=\Aut (\Q_0)$
of all circular order preserving permutations of $\Q_0.$  
Note that the action $G \times \Q_0 \to \Q_0$ is \textit{circularly ultrahomogeneous}, meaning that for every circular order preserving bijection between finite subsets $f \colon A \to B$ there exists $g \in G$ which extends $f$. 
Under the pointwise topology $G:=(\Aut (\Q_0), \t_0)$ is a Polish topological group. 
Choose $a_0 \in \Q_0$. Then the stabilizer 
$$H:=G_{a_0}=\{g \in \Aut (\Q_0): \ ga_0=a_0 \}$$ is a clopen  subgroup of $G$. It is topologically isomorphic to $\Aut(\Q,\leq)$. 
By an important well known result of V. Pestov \cite{Pest98}, this Polish group $\Aut(\Q,\leq)$ is extremely amenable.   

 \begin{remark} \label{r:counterexample} 
The subgroup  $\Aut(\Q,\leq)$ of $\Aut (\Q_0)$ is  inj-key but not  co-minimal.
 	\end{remark}
The careful explanation requires more definitions and several technical results.  
For simplicity we do not include the proof 
here. However, we are going to give 
the details (and some interesting related results and questions) in a separate article.

\sk

		\bibliographystyle{plain}


\end{document}